\documentclass[letterpaper, 12pt]{article}

\usepackage[margin=1in]{geometry}
 
\usepackage{tikz}

\usetikzlibrary{shadings}

\usepackage{amscd}

\usepackage{amsmath}

\usepackage{amsfonts}

\usepackage{amssymb}

\usepackage{amsthm}

\usepackage[utf8]{inputenc}
\usepackage[T1]{fontenc}
\usepackage[charter]{mathdesign} 

\usepackage{bbold}

\usepackage{mathtools}

\usepackage{arydshln}

\usepackage{fancyhdr}

\usepackage{verbatim}

\usepackage{enumitem}

\usepackage[normalem]{ulem}

\usepackage{color}

\usepackage[colorlinks=true, linkcolor=teal, citecolor=-teal,
pagebackref=true]{hyperref}

\usepackage[reftex]{theoremref}

\usepackage{nameref}

\usepackage{comment}

\usepackage[roundcorner=5pt,  skipabove=10pt, skipbelow=10pt]{mdframed}

\theoremstyle{definition}
\newmdtheoremenv[backgroundcolor=orange!10]{deminote}{Note}
\newmdtheoremenv[backgroundcolor=teal!10]{felipenote}{Note}
\newmdtheoremenv[backgroundcolor=pink!100]{manuelnote}{Note}

\usepackage[colorinlistoftodos]{todonotes}

\theoremstyle{plain}
\newtheorem*{theorem*}{Theorem}
\newtheorem{theorem}{Theorem}[section]
\newtheorem{lemma}[theorem]{Lemma}

\newtheorem{proposition}[theorem]{Proposition}

\newtheorem*{claim*}{Claim}

\theoremstyle{definition}

\newtheorem*{example*}{Example}

\theoremstyle{remark}
\newtheorem*{remark*}{Remark}

\newcommand{\ba}{\mathbf a}

\newcommand{\bA}{\mathbf A}

\newcommand{\balpha}{\boldsymbol \alpha}

\newcommand{\bbeta}{\boldsymbol \beta}

\newcommand{\bb}{\mathbf b}

\newcommand{\bc}{\mathbf c}

\newcommand{\bgamma}{\boldsymbol \gamma}

\newcommand{\bq}{\mathbf{q}}

\newcommand{\bu}{\mathbf{u}}

\newcommand{\bv}{\mathbf{v}}

\newcommand{\bx}{\mathbf x}

\newcommand{\by}{\mathbf y}

\newcommand{\RR}{\mathbb{R}}

\newcommand{\QQ}{\mathbb{Q}}

\newcommand{\NN}{\mathbb{N}}

\newcommand{\N}{\mathbb{N}}

\newcommand{\TT}{\mathbb{T}}

\newcommand{\ZZ}{\mathbb{Z}}

\newcommand{\calB}{\mathcal{B}}

\newcommand{\calD}{\mathcal{D}}

\newcommand{\hatD}{\widehat{D}}

\newcommand{\calH}{\mathcal{H}}

\newcommand{\calM}{\mathcal{M}}

\newcommand{\calQ}{\mathcal{Q}}

\newcommand{\bone}{\mathbb{1}}

\newcommand{\bzero}{\mathbf{0}}

\renewcommand{\leq}{\leqslant} \renewcommand{\geq}{\geqslant}

\renewcommand{\subset}{\subseteq}

\DeclarePairedDelimiter{\indcond}{[}{]}
\newcommand{\ind}[1]{\bone_{\indcond*{#1}}}

\DeclarePairedDelimiter{\abs}{\lvert}{\rvert}

\DeclarePairedDelimiter{\norm}{\lVert}{\rVert}

\DeclarePairedDelimiter{\ceil}{\lceil}{\rceil}
\DeclarePairedDelimiter{\set}{\lbrace}{\rbrace}
\DeclarePairedDelimiter{\parens}{\lparen}{\rparen}
\DeclarePairedDelimiter{\brackets}{\lbrack}{\rbrack}

\DeclareMathOperator{\supp}{supp}

\def\eps{{\varepsilon}}

\def\1int{{[0,1]}}

\title{The inhomogeneous Khintchine Theorem in dimension two}

\author{Demi Allen \\ University of Exeter \\ \texttt{d.d.allen@exeter.ac.uk} \and Manuel Hauke--Treuer \\ Graz University of Technology \\ \texttt{hauke@math.tugraz.at} \and Felipe A.~Ram{\'i}rez \\ Wesleyan University \\ \texttt{framirez@wesleyan.edu}}

\date{}

\begin{document}

\maketitle

\begin{abstract}
  We prove that the inhomogeneous variant of Khintchine's Theorem
  holds in dimension~$2$ without any monotonicity assumption. This
  resolves the last remaining case in the metric theory of
  inhomogeneous Diophantine approximation: while the monotonicity
  assumption is known to be unnecessary in dimensions $m\geq 3$ and
  necessary in dimension $m=1$, the two-dimensional case has remained
  open. It also settles the final outstanding case of a
  Khintchine--Groshev-type theorem for the approximation of systems of
  linear forms, confirming a conjecture of the first and third
  authors. Our results bring the inhomogeneous theory of metric
  Diophantine approximation into alignment with its homogeneous
  counterpart.
\end{abstract}

\tableofcontents

\section{Introduction and Statement of Results}

The problem addressed in this article concerns the simultaneous
approximation of two real numbers $\alpha_1$ and $\alpha_2$ by shifted
rational numbers $(u_1+\gamma_1)/q$ and $(u_2+\gamma_2)/q$ where
$\gamma_1, \gamma_2\in\RR$ are arbitrary fixed parameters and
$(u_1, u_2, q)\in\ZZ^2\times \NN$. Our main result, which we presently
state, completes a line of research that is inspired by Khintchine's
theorem (\cite[1924]{Khintchine}, \cite[1926]{KhintchineHD}) and has
been fully developed in all dimensions other than two in the
intervening decades. It is the seemingly harmless question of whether
a monotonicity condition can be removed from the classical theorem of
Khintchine and its inhomogeneous analogues.

To a function $\psi:\NN\to[0, \infty)$ and a fixed
$\bgamma = (\gamma_1, \gamma_2)\in \RR^2$ we associate the set
\begin{equation*}
  W^{\bgamma}(\psi) = \set*{\balpha \in [0,1]^2 : \abs{q \balpha - \bu -  \bgamma} < \psi(q) \textrm{ for infinitely many } (\bu,q) = (u_1,u_2,q) \in \ZZ^2 \times \NN},
\end{equation*}
where $\abs{\cdot}$ denotes the maximum norm in $\RR^2$. This is the
set of vectors in $\RR^2$ that are $\psi$-\emph{approximable with
  inhomogeneous parameter} $\bgamma$. (The restriction to $[0,1]^2$ in
the definition is a convenience afforded by the fact that the defining
condition is $\ZZ^2$-periodic.) Let $\lambda$ denote Lebesgue
measure. We prove the following theorem.

\begin{theorem}\th\label{thm:actualresult}
 Let $\psi:\NN\to [0,\infty)$ and let $\bgamma\in\RR^2$. Then
  \begin{equation*}
    \lambda(W^{\bgamma}(\psi)) =
    \begin{cases}
      0 &\textrm{if } \sum_{q=1}^\infty \psi(q)^2 < \infty, \\[2ex]
      1 &\textrm{if } \sum_{q=1}^\infty \psi(q)^2 = \infty.
    \end{cases}
  \end{equation*}
\end{theorem}

Previously, \th~\ref{thm:actualresult} was known to hold under
the added condition that $\psi$ is monotonic. For the moment, let us
write $W^{\bgamma}(\psi)=W_2^{\bgamma}(\psi)$. One can define the
corresponding set $W_m^{\bgamma}(\psi)$ in any dimension $m\geq 1$,
associated to some fixed $\bgamma\in\RR^m$. Khintchine proved that if
$\psi$ is monotonic and $\bgamma=\bzero$, then
\begin{equation}\label{eq:sum-measure}
  \lambda(W_m^{\bgamma}(\psi)) =
  \begin{cases}
    0 &\textrm{if } \sum_{q=1}^\infty \psi(q)^m < \infty, \\[2ex]
    1 &\textrm{if } \sum_{q=1}^\infty \psi(q)^m = \infty.
  \end{cases}
\end{equation}
where $\lambda$ now denotes Lebesgue measure in the appropriate
dimension.  This was subsequently extended to general $\bgamma$---in
1958 by Sz{\"u}sz for $m=1$ \cite{Szusz} and in 1964 by Schmidt for
$m > 1$~\cite{Schmidt}.

These results are often called metric dichotomies; they have a
``convergence part'' and a ``divergence part.''  The series
$\sum \psi(q)^m$ can be understood as a sum of measures, and by taking
this view, one can quickly prove the convergence part of the dichotomy
by a direct application of the Borel--Cantelli lemma. Since no
monotonicity assumption is needed in the application of
Borel--Cantelli, the question of that assumption's necessity in the
divergence part of the dichotomy quickly arose. In 1941, Duffin and
Schaeffer~\cite{DS} produced a famous example showing that
monotonicity is indeed needed in the case $m=1, \gamma=0$, and they
conjectured that a modified statement (requiring $\gcd(u,q)=1$) would
hold without any monotonicity assumption on $\psi$. The
Duffin--Schaeffer conjecture has its own rich history, culminating in
a proof by Koukoulopoulos and Maynard in 2020~\cite{KM}. The reader
may
consult~\cite{AistleitnerDS,ABH,ALMTZ,BHHVextraii,DS,ErdosDS,FH,HR,ManuelSurvey,HVW,HPVextra,KMY,KL,PV,RamirezDS,Vaaler}
for accounts of the story of the Duffin--Schaeffer conjecture and
natural continuations of it.

Meanwhile, Gallagher~\cite{Gallagher} showed in 1965 that monotonicity
can be removed from Khintchine's theorem when $m\geq 2$ and
$\bgamma=\bzero$. His proof works for general $\bgamma$ when
$m\geq 3$, though it was first stated explicitly this way in 1967 by
Ennola~\cite{Ennola}. In 1938, Groshev~\cite{Groshev} proved a dual
version of Khintchine's theorem where $\balpha\in\RR^n$ is viewed as a
column vector and the quantity to be minimized is
$\abs{\bq \balpha - u}$ where $\bq\in \ZZ^n$ is viewed as a row
vector. Again, a monotonicity assumption appears in the divergence
part of the theorem. The more general ``Khintchine--Groshev theory''
for systems of linear forms concerns
$\abs{\bq \balpha - \bu - \bgamma}$, where $\balpha$ is an $n\times m$
matrix with real entries, $(\bu,\bq)\in\ZZ^m\times\ZZ^n$, and
$\bgamma\in\RR^m$. Sprind{\v z}uk's 1979 book~\cite{Sprindzuk} proves
metric dichotomies in this context, carrying no monotonicity
assumption in the cases where $n\geq 3$ {and $\psi$ is univariate in
  the sense that it depends only on $\abs{\bq}$.}

The theme that emerges is that it is easier to prove the general
result when the dimensions are high, and more recent work has
continued to bear this out. In 2010, Beresnevich and
Velani~\cite{BVKG} completed the homogeneous (meaning
$\bgamma=\bzero$) Khintchine--Groshev theory, removing monotonicity
assumptions in all open dimensions other than $(m,n)=(1,1)$ in the
{univariate case, and in all dimensions where $m>1$ in the
  more general multivariate case where $\psi$ may depend freely on
  $\bq$. Duffin and Schaeffer had already shown through their example
  that monotonicity was required in dimension $(m,n)=(1,1)$, and their
  construction is quickly seen to work in all multivariate settings
  with $m=1$}. In 2017, the third author~\cite{counterexamples}
constructed analogous examples to show that in dimension one,
monotonicity is required for every inhomogeneous parameter
$\gamma$. Yu~\cite[2021]{Yu} then proved a suite of
  inhomogeneous results in the simultaneous settings (meaning $n=1$):
  for $m\geq 3$ he gave a modern proof of~(\ref{eq:sum-measure}) for
  general $\bgamma$ and with no extra assumption on $\psi$; for $m=2$
  he proved a version of~\th\ref{thm:actualresult} under an ``extra
  divergence'' assumption on the series. In 2023, the first
and third authors~\cite{AR} settled most of the remaining cases of the
inhomogeneous Khintchine--Groshev theory, leaving only $(n,m)=(1,2)$
and the univariate version of $(n,m)=(2,1)$
as conjectures. The latter was settled by the second
author~\cite[2025]{Hauke} for non-Liouville $\gamma$, and then
generally by Kim~\cite[2025]{Kim}. In that same paper, Kim also
settled the $(n,m)=(1,2)$ problem for
$\bgamma\in\QQ^2$, and in a recent paper~\cite{Kim_new}, he proved~\th\ref{thm:actualresult} under a polynomial decay condition on $\psi$. \th~\ref{thm:actualresult} completes the story by
proving the $(n,m)=(1,2)$ case for general $\bgamma\in\RR^2$ and general $\psi$. 

We will soon commence a more detailed discussion of these problems
and, in particular, of the obstacles preventing previous methods from
working in dimension two. But first, we present the most general
statements that come from combining what was previously known
with~\th~\ref{thm:actualresult}.

\subsubsection*{Combined general statements}

Let $m,n\in\NN$. To a function $\psi:\ZZ^n\to [0, \infty)$ and a
fixed $\bgamma\in\RR^m$ we associate the set 
\begin{equation*}
  W_{n,m}^{\bgamma}(\psi) = \set*{\balpha \in [0,1]^{n\times m}: \abs{\bq \balpha - \bu -  \bgamma} < \psi(\bq) \textrm{ for infinitely many } (\bu,\bq)\in \ZZ^m \times \ZZ^n},
\end{equation*}
where $\abs{\cdot}$ denotes the maximum norm in $\RR^m$ and
$[0,1]^{n\times m}$ is the set of all $n \times m$ matrices with real
entries in $[0,1]$.  The following is the most general multivariate
statement in the inhomogeneous Khintchine--Groshev theory.

\begin{theorem*}[Multivariate Inhomogeneous Khintchine--Groshev Theorem]
  Let $m,n\in\NN$ with $m>1$, let \mbox{$\bgamma\in \RR^m$}, and let $\psi:\ZZ^n\to [0,\infty)$. Then
  \begin{equation*}
    \lambda(W_{n,m}^{\bgamma}(\psi)) =
    \begin{cases}
      0 &\textrm{if } \sum_{\bq \in \ZZ^n} \psi(\bq)^m < \infty, \\[2ex]
      1 &\textrm{if } \sum_{\bq \in \ZZ^n} \psi(\bq)^m = \infty.
    \end{cases}
  \end{equation*}
  On the other hand, for every $n \geq 1$ and every $\gamma\in\RR$
  there exists \mbox{$\psi:\ZZ^n\to [0,\infty)$} such that
  $\sum \psi(\bq)$ diverges, while $\lambda(W_{n,1}^\gamma(\psi))=0$.
\end{theorem*}

\begin{remark*}
  The cases where $m=2$ were open for irrational $\bgamma$.
  \th~\ref{thm:actualresult} is the $(n,m) = (1,2)$ case, and the rest
  of the $(n,2)$ cases follow by
  applying~\cite[Theorem~1.2(a)]{bootstrap}
  with~\th~\ref{thm:actualresult} as input, i.e., by setting $m=2$ in
  the proof of~\cite[Theorem~2.3]{bootstrap}, which is the $m>2$
  version of the above theorem. The $m=1, n>1$ counterexample
  statements are immediate consequences of their $(n,m)=(1,1)$
  counterparts: define $\psi$ to be supported on vectors of the form
  $\bq = (q, 0, 0, \dots, 0)$ with $q\in\NN$, and put
  $\psi(\bq) = \psi_{DS}(q)$, where $\psi_{DS}:\NN\to[0,\infty)$ is a
  counterexample function from~\cite{DS,counterexamples}.
\end{remark*}

Historically, the univariate case, where the functions $\psi$ are
constant on spheres in $\ZZ^n$ with respect to the maximum metric, was
studied first. With Theorem~\ref{thm:actualresult} solving the
remaining case, the following theorem is the most general univariate
statement in the inhomogeneous Khintchine--Groshev theory.

\begin{theorem*}[Univariate Inhomogeneous Khintchine--Groshev Theorem]
  Let $m,n\in\NN$ with $mn>1$, let $\bgamma\in \RR^m$, and let
  $\Psi:\NN\to [0,\infty)$. For $\bq\in\ZZ^n\setminus\set{\bzero}$, put
  $\psi(\bq):=\Psi(\abs{\bq})$. Then,
  \begin{equation*}
    \lambda(W_{n,m}^{\bgamma}(\psi)) =
    \begin{cases}
      0 &\textrm{if } \sum_{q=1}^\infty q^{n-1}\Psi(q)^m < \infty, \\[2ex]
      1 &\textrm{if } \sum_{q=1}^\infty q^{n-1}\Psi(q)^m = \infty.
    \end{cases}
  \end{equation*}
  On the other hand, for every $\gamma\in\RR$ there
  exists $\psi:\mathbb{N}\to [0,\infty)$ such that $\sum \psi(q)$
  diverges, while $\lambda(W_{1,1}^\gamma(\psi))=0$.
\end{theorem*}

\subsubsection*{Hausdorff measure statements}

It is a matter of ongoing interest to establish metric dichotomies for
measures other than Lebesgue. The foundational result in this
direction is Jarnik's theorem~\cite{Jarnik}, which is the analogue of
Khintchine's theorem for Hausdorff measures.

Briefly, to a function $f:(0,\infty)\to(0,\infty)$ with
$f(x)\to 0\, (x\to 0)$, one associates a measure $\calH^f$, called the
Hausdorff $f$-measure, that is able to ``see'' sets on scales
determined by $f$. Often, a set of zero Lebesgue measure is best
understood through the use of an appropriate Hausdorff measure. For
example, the standard middle-thirds Cantor set, whose Lebesgue measure
is zero, is seen by the Hausdorff $f$-measure associated to the
function $f(x)=x^s$ where $s=\log 2/\log
3$. (See~\cite{FalconerBook}.)

In 2006, Beresnevich and Velani established the Mass Transference
Principle~\cite{BVmassTP}, a result that relates the Lebesgue measures
of limsup sets of balls to the Hausdorff measures of limsup sets of
suitably dilated balls. Using this, they showed that Khintchine's
theorem actually implies Jarnik's theorem. The Mass Transference
Principle has since become a standard tool for deducing metric
dichotomies for Hausdorff measures from corresponding Lebesgue-measure
statements, and many variants of it have appeared in the
literature. (See~\cite{AllenDaviaud} for an overview.)  One such
variant is~\cite[Theorem~1]{AllenBeresnevich}, a version of the Mass
Transference Principle for systems of linear forms. By applying it, we
deduce the following theorems from the general univariate and
multivariate statements above. The contributions coming
from~\th~\ref{thm:actualresult} are the cases where $m=2$.

\begin{theorem*}[Hausdorff Measure Multivariate Khintchine--Groshev
  Theorem]
    Let $m\geq 2$, and $\bgamma \in \mathbb{R}^m$ arbitrary. Let
    $f:\RR_{>0}\to \RR_{>0}$ and $g:r\to g(r):= r^{-m(n-1)}f(r)$ be
    dimension functions such that $r^{-nm}f(r)$ is monotonic.  Then
    for any function $\psi:\ZZ^n\to [0,\infty)$, we have
    \begin{equation*}
        \calH^f\left(W_{n,m}^{\bgamma}(\psi)\right) =
        \begin{cases}
            0 &\textrm{if } \sum\limits_{\bq\in \ZZ^n} \abs{\bq}^{m} g\parens*{\frac{\psi(\bq)}{\abs{\bq}}} <\infty, \\[2ex]
            \calH^f\left([0,1]^{mn}\right) &\textrm{if } \sum\limits_{\bq\in \ZZ^n} \abs{\bq}^{m} g\parens*{\frac{\psi(\bq)}{\abs{\bq}}} =\infty,
        \end{cases}
    \end{equation*}
    where $\calH^f$ denotes the Hausdorff $f$-measure.
\end{theorem*}

\begin{theorem*}[Hausdorff Measure Univariate Khintchine--Groshev Theorem]
  Let $mn\geq 2$, and $\bgamma \in \mathbb{R}^m$ arbitrary. Let
  $f:\RR_{>0}\to \RR_{>0}$ and $g:r\to g(r):= r^{-m(n-1)}f(r)$ be
  dimension functions such that $r^{-nm}f(r)$ is monotonic.  Then for
  any function $\Psi:\NN\to [0,\infty)$, we have
    \begin{equation*}
        \calH^f\left(W_{n,m}^{\bgamma}(\psi)\right) =
        \begin{cases}
            0 &\textrm{if } \sum\limits_{q\in \mathbb{N}} q^{n+m-1} g\parens*{\frac{\Psi(q)}{q}} <\infty, \\[2ex]
            \calH^f\left([0,1]^{mn}\right) &\textrm{if } \sum\limits_{q\in \mathbb{N}} q^{n+m-1} g\parens*{\frac{\Psi(q)}{q}} =\infty,
        \end{cases}
    \end{equation*}
    where $\calH^f$ denotes the Hausdorff $f$-measure, and
    $\psi(\bq):=\Psi(\abs{\bq})$.
\end{theorem*}

\subsubsection*{Asymptotic notation}

Throughout this paper, we use the notation $f\ll g$ to mean that there
exists a constant $C>0$ such that $f\leq C g$, and we use $g\gg f$
synonymously. We use the notation $f\asymp g$ to mean that $f\ll g$
and $f\gg g$ both hold.

\section{Heuristics and a model problem}\label{sec:heur-model-probl}

In this section, we explain the obstacles to proving
Theorem~\ref{thm:actualresult} and others like it, how those obstacles
were overcome in preceding works that solved the corresponding problem
in higher $n$-by-$m$-dimensions, and why the methods of those
preceding works fail in the $(n,m) = (1,2)$ case.  We also present a
model problem that highlights the key ideas of this manuscript and
serves to give an overview of the steps in the proof
of~\th~\ref{thm:actualresult}. 

 Most of the discussion takes place in the simultaneous
  setting (that is, where $n=1$), where the goal is to simultaneously
  approximate the components of a vector $\balpha\in\RR^m$.

\paragraph*{Quasi-independence on average.}

In all these Khintchine-type questions, one asks about the measure of
$\limsup_{q \to \infty}E_q$ where ${E_q:=}E_q(\psi)$ is a union of
small hypercubes. The first Borel--Cantelli lemma proves the
convergence part of the metric dichotomy immediately, and one is left
with the task of showing that the limsup set has full measure under
the divergence assumption. If we had perfect stochastic independence
of the set system $(E_q)_{q \in \mathbb{N}}$, then the second
Borel--Cantelli Lemma would immediately provide us with the
result. However, the system does not satisfy this strong independence
assumption, and the second Borel--Cantelli lemma is in practice
useless. Instead, there are refined versions of it that guarantee
positive measure of $\limsup_{q\to\infty} E_q$ under weaker
independence assumptions on $(E_q)_{q\in\NN}$. One commonly-used such
refinement---the one we use here---guarantees positive measure as long
as we can show the second moment bound
\begin{equation}\label{QIA}
  \sum_{q \neq r \leq Q}\lambda(E_q \cap E_r) \leq C \left(\sum_{q \leq Q}\lambda(E_q)\right)^2
\end{equation}
for infinitely many $Q \in \NN$, where $C > 0$ is an absolute
constant~\cite[Lemma~2.3]{Harman}.  In applications to Khintchine-type
problems, additional equidistributive properties of the sets $E_q$
allow one to strengthen the conclusion from positive measure to full
measure. Therefore, most of the effort goes toward
establishing~(\ref{QIA}), \emph{quasi-independence on average}.

\paragraph*{Overlap estimates.}

The above mentioned unions of hypercubes are
  \begin{equation}
    \label{eq:Eqdef}
    \begin{aligned}
      E_q&= \set*{\balpha \in [0,1]^m : \abs{q\balpha - \bu - \bgamma} < \psi(q) \textrm{ for some } \bu\in\ZZ^m} \\
         &= \bigcup_{\bu \in \ZZ^m} B\parens*{\frac{\bu+\bgamma}{q}, \frac{\psi(q)}{q}}\cap[0,1]^m
    \end{aligned}
  \end{equation}
  where $B(\bx, r)$ denotes the $\abs{\cdot}$-ball (a hypercube)
centered at $\bx$ having radius $r$.

In order to establish~(\ref{QIA}), it is crucial that we find
sufficiently sharp upper bounds for $\lambda(E_q \cap E_r)$, usually
known as \textit{overlap estimates}. An elementary calculation (see
e.g. \cite[Chapter~3]{Harman}) gives the overlap estimates
\begin{equation}\label{to_lattice_points}
\lambda(E_q \cap E_r) \ll \left(\min\left\{\frac{\psi(q)}{q},\frac{\psi(r)}{r}\right\}\right)^m \cdot \#\underbrace{\left\{(\bu ,\bv) \in \mathbb{Z}_q^m \times \mathbb{Z}_r^m: \left|\frac{\bu + \bgamma}{q} - \frac{\bv+\bgamma}{r}\right| < \frac{X}{qr}\right\}}_{C(q,r)},
\end{equation}
where $X:= X(q,r) = \max\set*{2r\psi(q), 2q\psi(r)}$ and
$\ZZ_q := \ZZ/(q\ZZ)$. The set $C(q,r)$ consists of points
$(\bu,\bv)\in\ZZ^m\times\ZZ^m$ lying in a region of volume $(2X)^m$,
which serves as a crude estimate of $\#C(q,r)$ with an error that is
$O(\gcd(q,r)^m)$.  If $X \geq \gcd(q,r)$, then this gives a bound
$\#C(q,r)\ll X^m$, which shows
$\lambda(E_q \cap E_r) \ll \lambda(E_q)\lambda(E_r)$ and would thus
imply~\eqref{QIA} if $X(q,r)$ were large for all $q,r$.

Once again, this is too much to hope for in practice.
Accepting that $X(q,r)<\gcd(q,r)$ occurs and that the
  $O(\gcd(q,r)^m)$ error is unavoidable, we get
\begin{equation}\label{trivial_overlap}
  \lambda(E_q \cap E_r) \ll \lambda(E_q)\lambda(E_r) + \left(\gcd(q,r)\right)^m \left(\min\left\{\frac{\psi(q)}{q},\frac{\psi(r)}{r}\right\}\right)^m,
\end{equation}
a bound that holds for all pairs
$(q,r) \in \NN \times
  \NN$. Applying~\eqref{trivial_overlap}, we obtain
\begin{equation*}
  \sum_{q,r \leq Q}\lambda(E_q \cap E_r) \ll \left(\sum_{q \leq Q}\lambda(E_q)\right)^2 + \sum_{r,q \leq Q} (\gcd(q,r))^m\parens*{\frac{\psi(q)}{q}}^{m}.
\end{equation*}
For $m>2$, elementary calculations show
\begin{equation}\label{eq:llqm}
  \sum_{r\leq q}(\gcd(q,r))^m\ll q^m,
\end{equation}
which, when combined with the previous expression, leads
to~(\ref{QIA}).
From here we can conclude for $m>2$ that if $\sum\psi(q)^m=\infty$,
then $W^{\bgamma}(\psi)$ has positive measure for every
$\bgamma\in \RR^m$. Full measure comes quickly thereafter, for example by applying of~\cite[Proposition~1]{AR}.
This recovers the nonmonotonic inhomogeneous Khintchine theorem in
dimension $m>2$ (due to Ennola~\cite{Ennola} and also found in
Yu~\cite[Theorem~1.8]{Yu}).

When $m=2$, we do not have~(\ref{eq:llqm}), rather, it comes with a
$\log\log q$ factor attached. The alternate bound
\begin{equation*}
  \sum_{q,r \leq Q}\lambda(E_q \cap E_r) \ll \left(\sum_{q \leq Q}\lambda(E_q)\right)^2 + \sum_{r,q \leq Q} \psi(q)^{m/2}\psi(r)^{m/2}\frac{\gcd(q,r)^m}{(qr)^{m/2}},
\end{equation*}
which is tighter in the terms where $q\psi(r) < r\psi(q)$, leads at
best to
\begin{equation}
  \label{eq:LR}
  \sum_{r,q \leq Q} \psi(q)^{m/2}\psi(r)^{m/2}\frac{\gcd(q,r)^m}{(qr)^{m/2}} \ll (\log \log Q)\sum_{q \leq Q}\psi(q)^m,
\end{equation}
and this is sharp; see Lewko--Radziwi\l\l~\cite{LR}.

Bounds of the strength of~(\ref{eq:LR}) with~(\ref{trivial_overlap})
have been used to prove~\th~\ref{thm:actualresult} under
extra-divergence assumptions (as
in~\cite[Theorem~1.8]{Yu},~\cite[Theorem~3]{AR}), but for arbitrarily
slowly growing $\sum_{q \leq Q}\psi(q)^2$, those bounds are not tight
enough to deduce~\eqref{QIA}.

Therefore, in the pursuit of~\th~\ref{thm:actualresult}, the
estimate~\eqref{trivial_overlap} should be avoided, since it has been
proven too weak to establish quasi-independence on average. Indeed, in
view of the counterexamples of Duffin and Schaeffer~\cite{DS} and the
inhomogeneous generalizations in~\cite{counterexamples} for $m = 1$,
this is not merely a technical obstruction, but possibly a structural
component of the problem: for certain $\psi$ and $\gamma$, the sets
$\set{E_q}_q$ in the one-dimensional problem are actually not
quasi-independent on average.  There is the possibility that this is
still the case for $m = 2$, an eventuality that would constitute an
actual obstruction to proving Theorem \ref{thm:actualresult} by
showing \eqref{QIA}.  A modified strategy presents itself: to find
subsets $E_q'\subset E_q$ having a divergent measure sum and for which
tighter overlap estimates can be established. The goal then becomes
proving \eqref{QIA} for sets $E_q'$ in place of $E_q$.

\paragraph*{Surgery on $E_q$.}

Gallagher's proof of the \textit{homogeneous} version
of~\th~\ref{thm:actualresult} used such a strategy. He defined the
sets $E_q'$ by applying a surgery to the sets $E_q$: excise the
hypercubes associated to non-primitive points, i.e., for $q\in\NN$,
define (from now on, let $m =
  2$)
\begin{equation}\label{eq:Eqprime}
  E_q' := \bigcup_{\bu \in S_q} B\parens*{\frac{\bu }{q}, \frac{\psi(q)}{q}}\cap[0,1]^2,
\end{equation}
where
$S_q := \{\bu = (u_1,u_2) \in \mathbb{Z}_q^2: \gcd(u_1,u_2,q) = 1\}$.
A routine computation shows that $\#S_q \gg q^2$ and therefore
$\lambda(E_q') \gg \lambda(E_q)$. (By contrast, when $m = 1$ the
corresponding surgery leads to a factor of $\varphi(q)/q$ in the new
measure. The monotonicity assumption in Khintchine's theorem ensures
that this new factor will not affect the divergence of the measure
sum, and the difficulty of the Duffin--Schaeffer Conjecture lay in
contending with this new factor for general functions $\psi$.)

The same reasoning that led to~\eqref{to_lattice_points} now gives 
\begin{equation}
\lambda(E_q' \cap E_r') \ll \left(\min\left\{\frac{\psi(q)}{q},\frac{\psi(r)}{r}\right\}\right)^2 \cdot \#\left\{(\bu,\bv) \in S_q \times S_r: \left|\frac{\bu}{q} - \frac{\bv}{r}\right| < \frac{X}{qr} \right\}.
\end{equation}
In the formerly problematic cases where $X(q,r)<\gcd(q,r)$, we now have
\begin{equation*}
  \left\{(\bu,\bv) \in S_q \times S_r: \left|\frac{\bu}{q} - \frac{\bv}{r}\right| < \frac{X}{qr} \right\} = \left\{(\bu,\bv) \in S_q \times S_r: \frac{\bu}{q} =
    \frac{\bv}{r}\right\},
\end{equation*}
which, by the primitivity assumptions $(\bu,\bv) \in S_q\times S_r$,
is empty when $q \neq r$. Therefore, the previous issue vanishes and
one can deduce \eqref{QIA} for $(E_q')_{q \in \mathbb{N}}$.

This approach has been generalized to allow for a \textit{rational}
inhomogeneous parameter $\bgamma = \bA/B = (A_1/B,A_2/B)$ (see e.g.~\cite{BHV,Kim}). Here, one replaces the set
$S_q$ in the definition of $E_q'$ with the set
\begin{equation}\label{rational_adapt}
  \{\bu = (u_1,u_2) \in \mathbb{Z}_q^2: \gcd(u_1B + A_1,u_2B + A_2,Bq) = 1\},
\end{equation}
{and the basic overlap estimates become
\begin{equation}\label{eq:basicestimatesinhom}
  \lambda(E_q' \cap E_r') \ll \left(\min\left\{\frac{\psi(q)}{q},\frac{\psi(r)}{r}\right\}\right)^2 \cdot \#\left\{(\bu,\bv) \in S_q \times S_r: \left|\frac{\bu+\bgamma}{q} - \frac{\bv+\bgamma}{r}\right| < \frac{X}{qr} \right\}.
\end{equation}
The primitivity condition in~(\ref{rational_adapt})} now forbids
solutions $(\bu, \bv)$ to
\begin{equation}\label{eq:nosols}
  \frac{\bu + \bgamma}{q} - \frac{\bv + \bgamma}{r} = 0,
\end{equation}
{so that
\begin{equation}
  \label{eq:atleast}
  \abs*{\frac{\bu+\bgamma}{q} - \frac{\bv+\bgamma}{r}} \geq \frac{\gcd(q,r)}{qrB},\quad (\bu, \bv)\in S_q\times S_r.
\end{equation}
This separation is used to bound the number of solutions counted
in~(\ref{eq:basicestimatesinhom}), leading to}
\begin{align}
\lambda(E_q' \cap E_r') &\ll
\lambda(E_q')\lambda(E_r') + 
    \frac{1}{D^2}\lambda(E_q')\lambda(E_r')\bone_{\brackets*{2D \geq 1/B}} \label{eq:rational_overlap} \\
        &\ll B^2 \cdot \lambda(E_q')\lambda(E_r'). \nonumber
\end{align}
Since $B$ is fixed, it can be absorbed into the implied constant
and \eqref{QIA} follows for $E_q'$.

Clearly, this approach would have to be modified in order for it to
apply in the case of irrational $\bgamma$. One natural modification is
to let the primitivity condition in~(\ref{rational_adapt}) depend on a
sequence $\set{\bA_k/B_k}_k$ of increasingly good rational
approximations to $\bgamma$. Such primitivity conditions were applied
by Schmidt in~\cite{Schmidt}, and also more recently by Chow--Technau
in~\cite{CT}, where they coined the term ``shift-reduced fractions''.
Reproducing the above discussion in this setting leads to
$\ll B_k^2 \lambda(E_q')\lambda(E_r')$. Since $B_k$ grows unboundedly,
the implicit constant does not absorb it, and~(\ref{QIA}) does not
immediately follow. Therefore, the strategy requires other
ingredients.

A recent paper of Kim~\cite{Kim_new} introduces one such ingredient in
the form of an extra assumption on $\psi$. Kim uses a shift-reduction
strategy to prove~(\ref{eq:sum-measure}) under the added assumption
that $\psi(q) = O(q^{-\delta})$ for some $\delta>0$. The effect of
this assumption is to introduce an upper bound on the quantity $X/qr$
which, when combined with his choice of shift-reduction, causes
$C(q,r)$ in~(\ref{eq:basicestimatesinhom}) to be empty when
$\gcd(q,r)> q^{1-\eps}$, for a suitably chosen $\eps>0$. Meanwhile,
the GCD sum~(\ref{eq:llqm}) is bounded $\ll q^2$ when restricted by
the condition $\gcd(q,r)\leq q^{1-\eps}$, and~(\ref{QIA}) follows.

The proof of~\th~\ref{thm:actualresult} also involves a
shift-reduction scheme, but before introducing it, let us describe a
more delicate view of the obstacle arising from small values of
$X(q,r)$.  This view, when combined with our shift-reductions, will
give crucial refinements of~(\ref{eq:rational_overlap}).

\paragraph*{A more delicate view.}

Recall that~(\ref{trivial_overlap}) comes from regarding
  the second term in the estimate
  \begin{equation*}
    \#C(q,r) \ll X^2 + \gcd(q,r)^2
  \end{equation*}
  as unavoidable. It is more accurate to say that it is not
  \emph{always} avoidable, and we profit from remembering that it only
  arises from solutions $(\bu, \bv)$ to
  \begin{equation}\label{eq:nearbycenters}
    \left| \frac{\bu + \bgamma}{q} - \frac{\bv + \bgamma}{r}\right| < \frac{X}{qr}
  \end{equation}
  when $X$ is much smaller than $\gcd(q,r)$. Such solutions are rare,
  and it is possible that for some $q,r$, there are none. Taking this possibility into account, a straightforward calculation
shows
\begin{equation}
  \label{eq:withindicator}
  \lambda(E_q\cap E_r) \ll \lambda(E_q)\lambda(E_r) + 
  \frac{1}{D^2}\lambda(E_q)\lambda(E_r)\bone_{\brackets*{\norm*{\frac{q-r}{\gcd(q,r)}\bgamma}<D}},
\end{equation}
where $D(q,r) := X(q,r)/\gcd(q,r)$ and $\norm{\cdot}$ denotes the
distance to the nearest integer
vector with respect to the maximum norm.

Dropping the indicator variable reproduces~\eqref{trivial_overlap},
which we know is too crude for our purposes in dimension $m = 2$. The
desired advantage from including the indicator comes from the
equidistribution of $(n\bgamma)_{n\in\NN}$ modulo one. A
  natural hope is that $\frac{q-r}{\gcd(q,r)}\bgamma \pmod 1$ sees that
  equidistribution and that the indicator takes a value $\ll D^2$ on
  average. Note however that $D:=D(q,r)$ depends on $\psi(q)$ and
$\psi(r)$ and may therefore behave erratically. Imposing a regularity
condition on $\psi$ (such as monotonic decrease or an
Erd\H{o}s--Vaaler-type condition as in Yu~\cite{Yu}) allows indeed to take advantage of equidistribution. But a general
  treatment requires us to allow for the possibility that the behavior
  of $\psi$ prevents the direct use of equidistribution.

  It turns out that the most decisive obstruction in non-monotonic
  approximations comes from the possibility that the \textit{support}
  of $\psi$ might be quite sparse. The recent paper~\cite{CHPR} formalizes this by producing
  Duffin--Schaeffer-type counterexamples that, unlike the classical
  examples, are monotonic on their support and in fact can have
  prescribed values on their support. This suggests that the
  \textit{regularity} of $\psi$ on its support is more of a technical
  burden, and we are thus motivated to study the following model
  problem which is free from that burden by design.

\paragraph*{A model problem.}

{We make the simplifying assumption that
  $\bgamma = (\gamma,\gamma)$ for some irrational $\gamma\in\RR$.}
Assume a large dyadic range $[T,2T]$ to be given, and let
$\mathcal{M} = \supp \psi \cap [T,2T]$ with
$\lvert \mathcal{M}\rvert = M$, where $M$ is a quantity that does not
depend on $T$. Now for $q\in \calM$, suppose
\begin{equation*}
  \psi(q) = \frac{1}{\sqrt{M}} \quad\textrm{so that}\quad \sum_{T \leq q \leq 2T}\psi(q)^2 = 1.
\end{equation*}
The second term in~(\ref{eq:withindicator}) now takes the form
  \begin{equation*}
    \frac{1}{D^2}\lambda(E_q)\lambda(E_r)\bone_{\brackets*{\norm*{\frac{q-r}{\gcd(q,r)}\bgamma}<D}} = \frac{1}{D^2M^2}\bone_{\brackets*{\norm*{\frac{q-r}{\gcd(q,r)}\gamma}<D}},
  \end{equation*}
  and the goal is to control it on average over the pairs
  $(q,r)\in[T, 2T]$. Note that only the pairs with $D(q,r)<1$ present
  any concern. We aim to show that
\begin{equation}\label{model_question}
  \sup_{\substack{\mathcal{M}\subseteq
      {[T,2T]}\\ |\mathcal{M}| = M }}
  \frac{1}{M^2} \sum_{j \in \mathbb{N}} 2^{2j} \sum_{\substack{q \neq r \in \mathcal{M}\\D(q,r) \leq 2^{-j}}}\bone_{\brackets*{{\norm*{\frac{q-r}{\gcd(q,r)}\gamma}}<2^{-j}}} \ll 1,
\end{equation}
with an implicit constant that neither depends on $T$ nor on $M$.

We make the following crucial observation: If $D(q,r) \leq 2^{-j}$,
then $\max\{\frac{q}{\gcd(q,r)}, \frac{r}{\gcd(q,r)}\} \leq 2^{-j}\sqrt{M}$,
and consequently $\frac{q-r}{\gcd(q,r)} \leq 2^{-j}\sqrt{M}$.  Note that
this is a bound in $M$, and not in $T$, which is key to having any
chance of proving \eqref{model_question}.  We now rewrite
\[
  \sum_{\substack{q\neq r \in \mathcal{M}\\D(q,r) \leq
      2^{-j}}}\bone_{\brackets*{{\norm*{\frac{q-r}{\gcd(q,r)}\gamma}}<2^{-j}}}
  \ll \sum_{\substack{1 \leq h \leq 2^{-j}\sqrt{M}\\ {\lVert
      h\gamma\rVert< 2^{-j}}}}f_j(h)
\]

where
\[f_j(h) := \#\left\{(q,r) \in \mathcal{M}: D(q,r) \leq 2^{-j},
    \frac{q-r}{\gcd(q,r)} = h\right\}.\] For fixed $q \in \mathcal{M}$,
there are at most $2^{-j}\sqrt{M}$ choices for $e:= \frac{r}{\gcd(q,r)}$,
and for fixed $e$, there is at most one $r$ solving the system
\[\frac{q-r}{\gcd(q,r)} = h, \quad \frac{r}{\gcd(q,r)} = e.\]
Thus we obtain $f_j(h) \ll M^{3/2} 2^{-j}$, which gives
    \begin{equation*}
      \sum_{\substack{1 \leq h \leq 2^{-j}\sqrt{M}\\
          \lVert h\gamma\rVert \leq 2^{-j}}}f_j(h) \ll M^{3/2} 2^{-j}\underbrace{\sum_{1 \leq h \leq 2^{-j}\sqrt{M}}\ind{\norm{h\gamma}< 2^{-j}}}_{\calB_j}.
    \end{equation*}
    It is vital to find a good estimate of the cardinality $\calB_j$
    of the Bohr set. Note that by simply dropping the condition
    $ \lVert h \gamma\rVert < 2^{-j}$, we obtain the Bohr set estimate
    $\calB_j \leq 2^{-j}\sqrt{M}$ for $j \ll \log M$ and $\calB_j = 0$
    for $j \gg \log M$. Plugging this bound into the
    {left}-hand side of \eqref{model_question}, gives
    \begin{equation*}
      \sup_{\substack{\mathcal{M}\subseteq [T,2T]\\ |\mathcal{M}| = M }} 
      \frac{1}{M^2} \sum_{j \in \mathbb{N}} 2^{2j} \sum_{\substack{q,r \in \mathcal{M}\\D(q,r) \leq 2^{-j}}}\bone_{\brackets*{\norm*{\frac{q-r}{\gcd(q,r)}\gamma}<2^{-j}}} \ll \sum_{j \ll \log M} 1,
    \end{equation*}
    which fails to prove \eqref{model_question}. 
    
    Evidently, we must
    win back a factor $\calB_j \leq 2^{-j}\sqrt{M}g(j)$ with
    $\sum_{j}g(j) \ll 1$ by making consequential use of the condition
    $\lVert h \gamma\rVert < 2^{-j}$. This can be achieved if $\gamma$
    has finite Diophantine exponent $\tau < \infty$, meaning
      that $\norm{h\gamma}\gg h^{-\tau}$ for all $h\geq 1$. In this
      case, the members of the Bohr set are pairwise separated by
      $\gg 2^{j/\tau}$ and we get
    $\calB_j \ll (2^{-j}\sqrt{M})2^{-j/\tau}$, leading to
    $g(j) := 2^{-j/\tau}$. {This solves the model problem
      for non-Liouville $\gamma$, and suggests} a strategy for proving
    Theorem \ref{thm:actualresult} for non-Liouville parameters
    $\bgamma\in\RR^2$.

    With a more careful analysis, one can use this same strategy to
    solve the model problem for parameters that are ``tamely
    Liouville,'' for example in the sense that
    $\lVert h \gamma \rVert \gg \exp(-\sqrt{h})$. Note however that
    for wildly Liouville $\gamma$, or in the extremal case rational
    numbers, the crude bound $\calB_j \ll 2^{-j}\sqrt{M}$ is all one
    can hope for.

    On the other hand, we have already described a strategy that
      works for rational parameters, namely, to impose the primitivity
      condition~(\ref{rational_adapt}) and thereby avoid the issues
      that arise from small $X(q,r)$. It turns out that the idea for
      irrational parameters of adapting~(\ref{rational_adapt}) to a
      sequence of increasingly good rational approximations becomes
      useful when combined with the more delicate
      estimate~(\ref{eq:withindicator}).

    \paragraph*{Solution to the model problem using shift reductions.}
    
  Continuing with the model problem, we define
  $A/B := A(M)/B(M)$ to be a rational number such that
  \begin{equation}\label{eq:proximity}
    \abs*{\gamma - \frac{A}{B}} < \frac{1}{\sqrt{M}},
  \end{equation}
  with $B$ minimal.  Since $\psi(q) = \frac{1}{\sqrt{M}}$,
  the triangle inequality shows that every solution
    $(\bu, \bv)$ to~(\ref{eq:nearbycenters}) is also a solution to
    \begin{equation}\label{eq:nearbyctrs2}
      \abs*{\frac{\bu + \bA/B}{q} - \frac{\bv + \bA/B}{r}} < \frac{2X}{qr},
    \end{equation}
    where $\bA = (A,A)$. 

  We now define $S_q$ as in \eqref{rational_adapt}. For the same reasons that lead to~\eqref{eq:rational_overlap}, we see that~(\ref{eq:nearbyctrs2}) can only have solutions
  $(\bu, \bv)\in S_q\times S_r$ if $2D(q,r)\geq 1/B$. Incorporating
this into the reasoning that led to~\eqref{eq:rational_overlap} and~(\ref{eq:withindicator}) gives
the refined estimate (see Lemma \ref{lem:overlaps})
  \begin{equation}\label{eq:refinedestimate}
    \lambda(E_q' \cap E_r') \ll \lambda(E_q')\lambda(E_r') + 
    \frac{1}{D^2}\lambda(E_q')\lambda(E_r')\bone_{\brackets*{\norm*{\frac{q-r}{\gcd(q,r)}\gamma}<D}}\bone_{\brackets*{2D \geq 1/B}},
  \end{equation}
  where the sets $E_q'$ are defined by~(\ref{eq:Eqprime}), now with
    the currently relevant, shift-reduced definition of $S_q$. This
  produces a refined version of the model
  problem~\eqref{model_question}, namely, we must show
  \begin{equation}
    \label{model_question2}
    \sup_{\substack{\mathcal{M}\subseteq [T,2T]\\ |\mathcal{M}| = M }} 
    \frac{1}{M^2} \sum_{j \in \mathbb{N}} 2^{2j} \sum_{\substack{q,r \in \mathcal{M}\\D(q,r) \leq 2^{-j}}}\bone_{\brackets*{\norm*{\frac{q-r}{\gcd(q,r)}\gamma}<2^{-j}}}\bone_{\brackets*{2^{j-1} \leq B(M)}} \ll 1.
  \end{equation}
  Importantly, the implied constant must be absolute, and is in
  particular not allowed to depend on $B$, since $B= B(M)$ grows with
  $M$. By following the arguments from above, one sees that
  establishing~(\ref{model_question2}) reduces to showing
  \begin{equation}
    \label{model_question3} 
    \frac{1}{\sqrt{M}} \sum_{\substack{j \in \mathbb{N}\\ 2^{j-1} \leq B}} 2^{j} 
    \sum_{1 \leq h \leq 2^{-j}\sqrt{M}}\ind{\norm*{h\gamma} < 2^{-j}}
    \ll 1.
  \end{equation}
  By~(\ref{eq:proximity}) and the triangle inequality, it suffices to show
  \begin{equation}\label{model_question3_rational} 
    \frac{1}{\sqrt{M}} \sum_{\substack{j \in \mathbb{N}\\ 2^{j-1} \leq B}} 2^{j} 
    \sum_{1 \leq h \leq 2^{-j}\sqrt{M}}\ind{\norm*{h\parens*{\frac{A}{B}}} < 2^{-j+1}}
    \ll 1.
  \end{equation}
  The Bohr set sum is now associated to a \textit{rational} rotation, and
  estimating it is a matter of understanding the distribution of
  rational rotations (compare Lemma \ref{bohrsetbound}).

  If $2^{-j}\sqrt{M} \geq B$, then
  $\set{h(A/B): 1 \leq h \leq 2^{-j}\sqrt{M}}$ is a full rotation
  orbit, and we have
  \begin{equation*}
    \sum_{1 \leq h \leq 2^{-j}\sqrt{M}}\ind{\norm*{h\parens*{\frac{A}{B}}} < 2^{-j+1}}
    \ll (2^{-j+1}B+1)\frac{2^{-j}\sqrt{M}}{B} \ll 2^{-2j}\sqrt{M},
  \end{equation*}
  recalling that in~(\ref{model_question3_rational}) we only consider $j$ for
  which $2^{j-1}\leq B$.
  It follows that
  \begin{equation*}
    \frac{1}{\sqrt{M}} \sum_{\substack{j \in \mathbb{N}\\ 2^{j-1} \leq B\\2^{-j}\sqrt{M} \geq B}} 2^{j} 
    \sum_{1 \leq h \leq 2^{-j}M}\ind{\norm*{h\parens*{\frac{A}{B}}} < 2^{-j+1}}
    \ll  \sum_{\substack{j \in \mathbb{N}\\ 2^{j-1} \leq B}} 2^{-j}
    \ll 1.
  \end{equation*}

  On the other hand, if $2^{-j}\sqrt{M} < B$, then
$\set{h(A/B): 1 \leq h \leq 2^{-j}\sqrt{M}}$ is a partial orbit omitting $0$. We
have
  \begin{equation*}
    \sum_{1 \leq h \leq 2^{-j}\sqrt{M}}\ind{\norm*{h\parens*{\frac{A}{B}}} < 2^{-j+1}}
    \leq \sum_{1 \leq h \leq B-1}\ind{\norm*{h\parens*{\frac{A}{B}}} < 2^{-j+1}} \ll B2^{-j+1}.
  \end{equation*}
  This estimate provides us with
  \begin{equation*}
    \frac{1}{\sqrt{M}}  \sum_{\substack{j \in \mathbb{N}\\ 2^{j-1} \leq B\\2^{-j}\sqrt{M} < B}} 2^{j} 
    \sum_{1 \leq h \leq 2^{-j}\sqrt{M}}\ind{\norm*{h\parens*{\frac{A}{B}}} < 2^{-j+1}}
    \ll
    \frac{B}{\sqrt{M}} \sum_{\substack{j \in \mathbb{N}\\ 2^{j-1} \leq B}} 1 \ll \frac{B\log B}{\sqrt{M}}.
  \end{equation*}
  If there were a fixed parameter $0 < \sigma < 1$ and a guarantee
  that $B \leq M^{\sigma/2}$, then this estimate would lead
  to~(\ref{model_question3_rational}), as desired. Fixing such a
  parameter, we are left only with the case where $B > M^{\sigma/2}$.
  
  In this case, we modify the shift-reduction in a way that can be
  seen as one of the main technical innovations in this article.
  Rather than shift-reducing with respect to $\bA/B$, we shift-reduce
  with respect to $\ba/b$, where $1 \leq b < B$ and
  \begin{equation}\label{approx_b}
    b\left|\gamma - \frac{a}{b}\right| = \norm{b\gamma} \leq \frac{1}{B} < M^{-\sigma/2}
  \end{equation}
  with $\ba = (a,a)$. The rational number $a/b$ is guaranteed by
  Dirichlet's theorem to exist. We now have no solutions $(\bu,\bv)$
  to~(\ref{eq:nosols}) (having replaced $(\bA,B)$ with $(\ba, b)$),
  which implies~(\ref{eq:atleast}) (having made the same
  replacement). Meanwhile, the triangle inequality shows that any
  solution $(\bu, \bv)\in S_q\times S_r$ to~(\ref{eq:nearbycenters}) is also a solution
  to
  \begin{equation*}
    \abs*{\frac{\bu + \ba/b}{q} - \frac{\bv + \ba/b}{r}} < \frac{X}{qr} + \abs*{\frac{(q-r)\norm{b\gamma}}{qrb}}
  \end{equation*}  
  which, when combined with~(\ref{eq:atleast}) (with $\ba,b$ instead
  of $\bA,B$), gives
  \begin{equation*}
    \frac{\gcd(q,r)}{qrb} \leq \frac{X}{qr} + \abs*{\frac{(q-r)\norm{b\gamma}}{qrb}}.
  \end{equation*}
  Multiplying by $qrb/\gcd(q,r)$ gives
  \begin{equation*}
    1 \leq Db + \abs*{\frac{q-r}{\gcd(q,r)}}\norm{b\gamma} \leq Db + D\sqrt{M} \norm{b\gamma}.
  \end{equation*}
  Incorporating this into the reasoning
  that led to~(\ref{eq:refinedestimate}) gives
  \begin{equation}\label{eq:refinedestimate_ab}
    \lambda(E_q' \cap E_r') \ll \lambda(E_q')\lambda(E_r') + 
    \frac{1}{D^2}\lambda(E_q')\lambda(E_r')\bone_{\brackets*{\norm*{\frac{q-r}{\gcd(q,r)}\gamma}<D}}\ind{1 \leq Db + D\sqrt{M}\norm{b\gamma}},
  \end{equation}
  which will lead to a newly refined model problem.

Splitting into the cases $D\sqrt{M} \lVert b \gamma\rVert \leq \frac{1}{2}$
and $D\sqrt{M} \lVert b \gamma\rVert > \frac{1}{2}$ gives
\begin{align*}
  \ind{1 \leq Db + D\sqrt{M}\norm{b\gamma}}
  &\leq \ind{D\sqrt{M} \norm{b\gamma} \leq 1/2}\ind{2D\geq 1/b} + \ind{D\sqrt{M} \norm{b\gamma} > 1/2},
    \intertext{and fixing a small parameter $0 < \rho < \sigma$, we find}
  &\leq \ind{2b > D^{-\rho}} + \ind{D\sqrt{M} \norm{b\gamma} > 1/2}\ind{2b \leq  D^{-\rho}}.
\end{align*}
As before, we stratify over dyadic values of $D$ to arrive at the
model problem we need to solve when $B > M^{\sigma/2}$. It is
\begin{equation}
\label{model_question4}\sup_{\substack{\mathcal{M}\subseteq [T,2T]\\ |\mathcal{M}| = M }} 
\frac{1}{M^2} \parens*{\sum_{\substack{j \in \mathbb{N}\\ 2b > 2^{\rho j}}} 2^{2j} \sum_{\substack{q,r \in \mathcal{M}\\D(q,r) \leq 2^{-j}}}\bone_{\brackets*{\norm*{\frac{q-r}{\gcd(q,r)}\gamma}<2^{-j}}}
+   \sum_{\substack{j \in \mathbb{N}\\ 2b \leq 2^{\rho j}\\
    2^j < \sqrt{M}\lVert b \gamma \rVert
  }} 2^{2j} \sum_{\substack{q,r \in \mathcal{M}\\D(q,r) \leq 2^{-j}}}\bone_{\brackets*{\norm*{\frac{q-r}{\gcd(q,r)}\gamma}}<2^{-j}}}
\ll 1.
\end{equation}
Again, the treatment of this problem requires suitable 
estimates on 
\begin{equation*}
  \sum_{1 \leq h \leq 2^{-j}\sqrt{M}}\ind{\norm*{h\gamma} < 2^{-j}}.
\end{equation*}

Depending on the case, we employ different bounds for the Bohr sets in
the discussion (compare \th\ref{adhocbohrsetbounds}): Decomposing $h = mb + s$ and using \eqref{approx_b},
we note that
\begin{equation}\label{decomp_heur}
  \#\{1 \leq h \leq 2^{-j}\sqrt{M}: \lVert h \gamma\rVert \leq 2^{-j}\}
  \leq \sum_{1 \leq m \leq \frac{2^{-j}\sqrt{M}}{b}}\sum_{1 \leq s \leq b} \ind{\norm*{s\parens*{\frac{a}{b}} + mb\gamma} \leq (2^{-j} + M^{-\sigma/2})}.
\end{equation}
In the case $2b>2^{\rho j}$, we bound this further from above by 
\begin{equation*}
    \begin{split}
      \frac{2^{-j}\sqrt{M}}{b} \underbrace{\sup_{\beta \in \mathbb{R}} \#\{1 \leq s \leq b: \norm*{s\parens*{\frac{a}{b}} + \beta} \leq (2^{-j} + M^{-\sigma/2})\}}_{\ll b(2^{-j} + M^{-\sigma/2}) + 1}\ll 2^{-2j}\sqrt{M} + 2^{-j}M^{(1-\sigma)/2} + \frac{2^{-j}\sqrt{M}}{b}.
   \end{split}
\end{equation*}
This yields that 
\begin{equation*}
  \frac{1}{\sqrt{M}} \sum_{\substack{j \in \mathbb{N}\\ 2^j \leq
      \sqrt{M}\\ 2b > 2^{\rho j}}} 2^{j} \sum_{\substack{1 \leq h \leq
      2^{-j}M\\ \lVert h\gamma\rVert \leq 2^{-j}}} 1 \ll 1 +
  \sum_{\substack{j \in \mathbb{N}\\ 2^j \leq \sqrt{M}}}M^{-\sigma/2}
  + \sum_{\substack{j \in \mathbb{N}\\ b > 2^{\rho j}}}\frac{1}{b} \ll
  1 + \frac{\log M}{M^{\sigma/2}} + \frac{\log b}{b}\ll 1,
\end{equation*}
bounding the first sum in~(\ref{model_question4}).  If
$b \leq 2^{\rho j}$ and $2^{-j}\sqrt{M}\lVert b \gamma \rVert > 1/2$, then we bound \eqref{decomp_heur} using spacing
arguments (compare Lemma \ref{lem:oncearound}) by
\begin{equation}
    \begin{split}
&\leq \sum_{1 \leq s \leq b} \left(\frac{\sqrt{M}2^{-j} \lVert b \gamma\rVert}{b} + 1\right) \sup_{\beta \in \mathbb{R}}
\#\left\{1 \leq n \leq \frac{1}{\lVert b \gamma \rVert}: \lVert n\cdot (b\gamma)+ \beta \rVert \leq (2^{-j} + M^{-\sigma/2})\right\}
\\&\ll \left(\sqrt{M}2^{-j} \lVert b \gamma\rVert + b\right) 
\left({\frac{(2^{-j} + M^{-\sigma/2})}{\lVert b \gamma \rVert }}+1\right).
   \end{split}
\end{equation}
Since $\lVert b \gamma \rVert < \frac{1}{B} < M^{-\sigma/2}$ and
$\frac{1}{\lVert b\gamma \rVert} \ll \sqrt{M}2^{-j}$, the above is bounded by
$b2^{-2j}\sqrt{M} + 2^{-j}bM^{(1-\sigma)/2}$, and we get
\begin{equation*}
  \frac{1}{\sqrt{M}} \sum_{\substack{j \in \mathbb{N}\\ 2b \leq 2^{\rho j}\\
    2^j < \sqrt{M}\lVert b \gamma \rVert
  }} 2^{j} \sum_{\substack{1 \leq h \leq 2^{-j}M\\ \lVert
      h\gamma\rVert \leq 2^{-j}}} 1 \ll \sum_{\substack{j \in
      \mathbb{N}\\ 2^j < \sqrt{M}\\b \leq 2^{\rho j}}}b M^{-\sigma/2} + \sum_{\substack{j
      \in \mathbb{N}\\b \leq 2^{\rho j}}}b2^{-j} \ll M^{\rho/2 - \sigma/2} +
  \sum_{\substack{j \in \mathbb{N}}}2^{-j(1 - \rho)} \ll 1.
\end{equation*}
This finally covers all the cases, giving an affirmative answer to the
refined model problems of \eqref{model_question3} in the cases
$B\leq M^{\sigma/2}$ and \eqref{model_question4} in the cases
$B> M^{\sigma/2}$.

\paragraph*{In summary of the solution of the model problem.}

The treatment of the model problem can be summarized as follows. Once
we have shift-reduced, we obtain the overlap estimate
\begin{equation*}
  \lambda(E_q' \cap E_r') \ll \lambda(E_q')\lambda(E_r') + 
  \frac{1}{D^2}\lambda(E_q')\lambda(E_r')\bone_{\brackets*{\norm*{\frac{q-r}{\gcd(q,r)}\gamma}<D}}\bone_{\brackets*{2D
      \geq 1/B}},
\end{equation*}
which holds regardless of $\gamma$.

If $\gamma$ were rational, then we could drop the first indicator
$\bone_{\brackets*{\norm*{\frac{q-r}{\gcd(q,r)}\gamma}<D}}$, and this
would lead to $\ll B^2 \lambda(E_q')\lambda(E_r')$ as discussed. But
this would not work for irrational $\gamma$.

If $\gamma$ is a non-Liouville irrational, then one can drop
$\bone_{\brackets*{2D \geq 1/B}}$ instead, and a sufficient amount of
saving comes from the rate of equidistribution of $q\gamma \pmod 1$,
as described. In fact, this strategy is viable even without
shift-reductions. But it does not work for Liouville $\gamma$, even
with the shift-reductions.

For Liouville $\gamma$, we cannot drop either of the indicators, but
instead must use both simultaneously. Each indicator has a different
regime where---through the fine-tuning of the shift-reduction
scheme---it unfolds its full power. The different regimes cover the
full range of possibilities, and, in particular, this strategy works
equally well for non-Liouville $\gamma$ as it does for Liouville
$\gamma$. Indeed, if $\gamma$ has Diophantine type
$1\leq\tau < \infty$, then one has
\begin{equation}\label{eq:instructive}
  M^{\frac{1}{2}\parens*{\frac{1}{1+\tau}}}\ll B \ll   M^{\frac{1}{2}\parens*{\frac{\tau}{1+\tau}}}.
\end{equation}
If we select $\tau/(1+\tau) < \sigma < 1$, then the model problem can
be solved by shift-reducing with respect to $A/B$, for all $M$; If we
select $0 < \sigma < 1/(1+\tau)$, it can be solved by shift-reducing
with respect to $a/b$, for all $M$, thus actually both approaches
work. A Liouville parameter $\gamma$, on the other hand, has
Diophantine type $\tau=\infty$, so we lose~(\ref{eq:instructive}), and
the successful shift-reduction is forced to depend on $M$, regardless
of the choice of $\sigma$.  We select a fixed value for
$\sigma\in (0,1)$, that can be chosen arbitrarily close to
$1$. Whenever $B \leq M^{\sigma/2}$ does not hold, there exists
$b\leq M^{(1-\sigma)/2}$ such that
$\lVert b\gamma \rVert \leq M^{-\sigma/2}$ which after
shift-reduction, results in an indicator
\begin{equation*}
  \ind{1 \leq Db + D\sqrt{M}\norm{b\gamma}}\leq \bone_{\brackets*{D
      \geq M^{(\sigma-1)/2}}}
\end{equation*}
that is very powerful, and sufficient for establishing the result.

\paragraph*{Beyond the model problem.}

Of course, the model problem is a much simplified version of the main
problem behind~\th~\ref{thm:actualresult}, and several complications
arise once we exit its confines. First, we cannot in general assume
that $\psi$ is constant on dyadic ranges in its support; second, we
cannot assume that the inhomogeneous parameter takes the form
$\bgamma=(\gamma, \gamma)$ for some $\gamma\in \RR$; and third, it is
not sufficient to only bound the overlaps $\lambda(E_q\cap E_r)$ for
pairs $q,r$ lying in dyadic ranges $[T,2T]$.

\

Since $\psi$ is not necessarily constant on dyadic ranges on its
support, the definition of $B$ must depend on $\psi(q)$ (and not only
on $\#\mathcal{M}= M$). Consequently, the shift
  reductions defining $E_q'$ and $E_r'$ are generally with respect to
  different rationals $\bA_q/B_q$ and $\bA_r/B_r$, even for $q,r$ in
  the same dyadic range.  {The above discussion's
  definitions are adapted in~(\ref{eq:Bkbound}) and~(\ref{eq:bk}), and
  the shift-reduction becomes~\eqref{shift_reduced_sets}}---the reader
will note a distinction
\begin{equation*}
  B_k \leq 2^{\sigma k}\quad\textrm{versus}\quad
  B_k >2^{\sigma k}.
\end{equation*}
This is analogous to the similar distinction made in the model
problem, and leads to estimates analogous to those for solving model
problems~\eqref{model_question2} and~\eqref{model_question3},
respectively.  A subtlety that arises from the dependence of
$\bA_q/B_q$ on $\psi(q)$ is that for $r < q$, it might happen that
$B_r \neq B_q$ and that the estimates~(\ref{eq:basicestimatesinhom})
count solutions $(\bu, \bv)$ to
\begin{equation*}
  \frac{B_r\bu+\bA_r}{B_rq} - \frac{B_r\bv+\bA_r}{B_rr} = 0,\quad  (\bu,\bv) \in S_q \times S_r.
\end{equation*}
By the definition of $S_r$, the latter rational vector is in reduced
form, but the former might not be. Thus the above may produce solutions
when $q$ is an integer multiple of $r$. However, it turns out that
pairs with $r \mid q$ can be estimated by the trivial estimate
\eqref{trivial_overlap}, and the corresponding GCD-sums give an
acceptable contribution.

\

The model problem's assumption that $\bgamma = (\gamma, \gamma)$ had
the effect of reducing much of the discussion to the one-dimensional
Diophantine approximation of $\gamma\in \RR$. In the genuinely
two-dimensional case $\bgamma=(\gamma_1, \gamma_2)$, one sees two main
differences. First, the rational vectors $\ba/b$ and $\bA/B$ now arise
as simultaneous good approximations $(a_1/b,a_2/b)$, $(A_1/B,A_2/B)$
to $(\gamma_1,\gamma_2)$---see \eqref{eq:Bkbound} and
\eqref{eq:bk}. As in the model problem, $\bA/B$ is a ``best
approximation of the first kind,'' while $\ba/b$ comes from an
application of (the simultaneous version of) Dirichlet's Theorem. Of
course, in the two-dimensional setting we have $\norm{b\bgamma} \leq B^{-1/2}$
instead of $\norm{b\gamma} \leq B^{-1}$. The second difference is that the Bohr set
estimates, which in the model problem required only the understanding
of rotation orbits in the circle, now require the consideration of
orbits of toral translations in dimension two.

Fortunately, the strategy described above has enough room to
accommodate both of these changes---the lower approximation exponent
for $\ba/b$ and the higher-dimensional toral translations. The Bohr
set estimates are carried out in Section~\ref{sec:bohr-set-bounds}. In
fact, one sees that the worst-case scenarios were already present for
diagonal parameters $\bgamma=(\gamma, \gamma)$, and the genuinely
two-dimensional parameters do not cause much additional trouble.

\

Finally, since we cannot assume that
$\sum_{n \in [T,2T]}\psi(n)^2 \geq 1$ for infinitely many $T$, it is
not enough to only bound overlaps for pairs $q,r$ with $q \asymp
r$. We also must treat pairs lying in different dyadic ranges,
$q \in [Q, 2Q]$ and $r\in [R,2R]$, possibly with $Q/R$ being large. It turns
out that we need to win some extra factor in terms of $Q/R$, since
otherwise, the dyadic approach considered here loses too much. This
gain comes for different reasons in different cases, depending on the
relative sizes of $\psi(r), \psi(q)$, but can be explained by the
effect that the maximum condition in the definition of $D(q,r)$ (which
helps in this case since the ``enemy'' is $D$ being small) has. It
makes asymmetric situations easier to handle, and allows for the
desired additional gain (it is actually almost linear in $Q/R$). A
final inconvenience arises in this setup since replacing $\bgamma$
with its rational approximation $\bA_q/B_q$ is no longer
straightforwardly possible: For $Q \asymp R$, the approximation error
was bounded by $D$, which can be compensated by replacing $D$ in the
Bohr sets with $2D$. In the new setup, the error is in fact
$\frac{Q}{R}D$ (to be later denoted by $\widehat{D}$). This issue is
treated by yet another case distinction, depending on whether
$\frac{Q}{R}D < D^{\tau}$ or not, where $0 < \tau < 1$ is yet another
parameter for us to choose; If $\frac{Q}{R}D < D^{\tau}$, then we
simply increase the window of the Bohr set to $D^{\tau}$, and the
estimates presented above still have enough flexibility to accommodate
this additional loss in the exponent. If $\frac{Q}{R}D > D^{\tau}$,
then this implies $D > (R/Q)^{1/(1-\tau)}$, and we can only have
$\ll \log Q/R$ many dyadic ranges for such $D$. This factor is
overcompensated by the almost linear gain $R/Q$ discussed before.

\

\paragraph*{Remarks on linear forms.}

We end this section with some remarks on the $n$-by-$m$ linear forms
cases, where $n\geq 2$. Recall that the initial difficulty in the
pursuit of~\th~\ref{thm:actualresult} is that the
bound~(\ref{eq:llqm}) on GCD sums fails for $m=2$. On the other hand,
when $m\geq 3$, that bound together with the trivial overlap
estimates~(\ref{trivial_overlap}) gives quasi-independence on average
without having to study the Diophantine properties of the
inhomogeneous parameter $\bgamma$---the problem reduces to
measure-theoretic considerations. Similar phenomena occur in all
dimensions $(n,m)$ other than $(2,1)$ or $(1,2)$. In the results of Sprind{\v
  z}uk~\cite{Sprindzuk}, Beresnevich--Velani~\cite{BVKG}, and
of the first and third authors~\cite{AR}, one sees measure-theoretic arguments
that are facilitated in large part by controlling the average behavior
of quantities involving elementary arithmetic functions, and not
because of any precise study of Diophantine properties of
parameters. Indeed, the classical theorems, with their monotonicity
assumptions, are proved similarly. In~\cite{bootstrap} it is shown
that, for measure-theoretic and geometric reasons, every
Khintchine-type problem that can be solved in dimension $(1,m)$ can
automatically be solved multivariately in dimension $(n,m)$ for all
$n\geq 1$. (In particular, the
$(1,2)$-dimensional~\th~\ref{thm:actualresult} implies the $(n,2)$
cases of the general multivariate Khintchine--Groshev theorem stated
above.)

The dual case $(n,m) = (2,1)$ is different in this regard. It had a
similar issue arising from divergent GCD sums, causing trivial overlap
estimates to be insufficient for quasi-independence on average. There,
a comparable, but much more straightforward surgery of the sets was
applied: Since the measure sum in that case reads
$\sum_{q}q\psi(q) = \infty$, one can assume without loss of generality
that $\psi(q) \leq 1/q$. This allowed for a better control over the
size of $D$ in terms of $q$ (and not only in $\psi(q)$), which then
allowed for the application of (slightly tilted) GCD sum bounds (see
the papers~\cite{Hauke,Kim}). (This is similar to the advantage gained
in Kim's recent paper~\cite{Kim_new} where he makes the assumption of
$\psi(q)=O(q^{-\delta})$, as mentioned above.)  This settled the case
$(n,m) = (2,1)$, without having to appeal to equidistributional
properties of $(q\gamma)_{q \in \mathbb{N}}$, and the complicated case
distinction that appears to be necessary in the final case
$(n,m) = (1,2)$ considered in this article.

\section{Preliminaries}

The aim of this section is to set the notation, introduce the suitable shift-reduction, and reduce the question indeed to the problem described
above in establishing quasi-independence on average.  Given
$\psi:\NN\to [0,\infty)$ and $\bgamma\in\RR^2$, define
\begin{equation*}
  E_q= E_q(\psi) = \set*{\balpha \in [0,1]^2 : \norm{q\balpha - \bgamma} < \psi(q)},
\end{equation*}
and note that this definition coincides with~(\ref{eq:Eqdef}). Then
$W^{\bgamma}(\psi) = \limsup_{q\to\infty} E_q$, and this
characterization allows us to treat the convergence case in Theorem
\ref{thm:actualresult} immediately.

In fact, the following proposition allows us to focus our attention on
the divergence case, with functions $\psi$ that approach $0$ and
irrational inhomogeneous parameters.
  
\begin{proposition}\th\label{prop:wlog}
  If
  \begin{enumerate}
  \item $\sum_{q \in \mathbb{N}} \psi(q)^2$ converges, or
  \item $\limsup_{q \to \infty} \psi(q) > 0$, or
  \item $\bgamma\in\QQ^2$,
  \end{enumerate}
  then~\th\ref{thm:actualresult} holds.
\end{proposition}

\begin{proof}
  The first assertion is a direct application of the Borel--Cantelli
  lemma. From~(\ref{eq:Eqdef}) we obtain $\lambda(E_q) \leq 4\psi(q)^2$,
  and by the convergence assumption, we have
  $\sum_{q \in \mathbb{N}} \lambda(E_q) < \infty$. Thus, the convergence
  Borel--Cantelli Lemma proves the claim.

  For the second assertion, assume
  $\limsup_{q \to \infty} \psi(q) > 0$. Then there exists a sequence
  $\mathcal{Q} = (q_i)_{i \in \mathbb{N}}$ of increasing integers such
  that $\psi(q_k) > \delta$ for some $0 < \delta < \frac{1}{2}$. We
  now define
  \begin{equation*}
    \tilde{\psi}(q)
    :=
    \begin{cases}
      \delta \text{ if } q \in \mathcal{Q},\\
      0 \text{ otherwise. }
    \end{cases}
  \end{equation*}
  Note that we still have
  $\sum_{q \in \mathbb{N}} \tilde{\psi}(q) = \infty$, and since
  $\tilde{\psi} (q) \leq \psi(q)$ for all sufficiently large $q$, we
  have $W^{\gamma}(\tilde{\psi}) \subseteq W^{\gamma}(\psi)$. It
  therefore suffices to show $\lambda(W^{\gamma}(\tilde{\psi})) = 1$.
  Let $U\subset [0,1]^2$ be an open ball. With
  \begin{equation*}
    E_q := \bigcup_{\bu \in \mathbb{Z}_q^2} B\parens*{\frac{\bu + \bgamma}{q},
      \frac{\tilde{\psi}(q)}{q}}\cap[0,1]^2
  \end{equation*}
  and the equidistribution of $\{\bu/q: \bu \in \mathbb{Z}_q^2\}$ in
  $[0,1)^2$, we immediately obtain
  \begin{equation*}
    \lambda(E_q \cap U) \geq \frac{1}{2}\lambda(E_q)\lambda(U)\geq
    2\delta^2\lambda(U)
  \end{equation*}
  for all sufficiently large $q\in\calQ$. Therefore,
  \begin{equation*}
    \lambda(\limsup_{q\to\infty} E_q\cap U) = \lambda(W(\tilde\psi)\cap U) \geq 2\delta^2\lambda(U).
  \end{equation*}
  Since this holds for all open balls $U$, the Lebesgue density
  theorem gives $\lambda(W(\tilde\psi))=1$ as desired.

  The third assertion is a result of Kim~\cite[Theorem~2]{Kim}.
\end{proof}

In view of~\th~\ref{prop:wlog} we will assume for the rest of this
article the following:
\begin{align}
    \label{divergent}&\sum_{q \in \mathbb{N}} \psi(q)^2 = \infty,\\
    \label{psi_to_0}&\lim_{q \to \infty}\psi(q) = 0,\\
    \label{irrational}&\bgamma \notin \mathbb{Q}^2.
\end{align}
Furthermore, by modifying the value of $\psi(q)$ for at most finitely
many $q\in\NN$, we may further assume that $\psi:\NN\to[0,1/2]$.

\subsection{Shift reduction}\label{sec:shift-reduction}

Let $\bgamma\in\RR^2 \setminus{\mathbb{Q}^2}$ be fixed. For $k\geq 0$, let $B_k \geq 1$ be the smallest
integer such that there exists $\bA_k\in \ZZ^2$ with
\begin{equation}\label{eq:Bkbound}
  \abs*{\bgamma - \frac{\bA_k}{B_k}} = \frac{\norm{B_k\bgamma}}{B_k}< 2^{-k}.
\end{equation}
By minimality of $B_k$, we have that $\gcd(\bA_k, B_k)=1$. If
$B_k\geq 2$, let $1 \leq b_k < B_k$ be minimal such that there exists
$\ba_k\in\ZZ^2$ with
\begin{equation}\label{eq:bk}
  \abs*{\bgamma - \frac{\ba_k}{b_k}} = \frac{\norm{b_k\bgamma}}{b_k} \leq \frac{1}{b_k B_k^{1/2}},
\end{equation}
as furnished by Dirichlet's theorem.  

The inequality
\begin{equation}\label{eq:bkbound}
  b_k \leq   \frac{2^k}{B_k^{1/2}},
\end{equation}
follows from the minimality of $B_k$. Again, we have
$\gcd(\ba_k, b_k)=1$.

To $\psi:\NN\to [0, 1/2]$ and $0 < \sigma < 1$ we associate the
following shift-reduced rational points. For $q$ such that
$2^{-k} < \psi(q) \leq 2^{-k+1}$, define
\begin{equation}\label{shift_reduced_sets}
  S_q =
  \begin{cases}
    \set*{\bu \in (\ZZ/q\ZZ)^2 : \gcd(B_k\bu + \bA_k, B_kq) =1} &\textrm{if } B_k\leq 2^{\sigma k},\\
    \set*{\bu \in (\ZZ/q\ZZ)^2 : \gcd(b_k\bu + \ba_k, b_kq) =1} &\textrm{if } B_k> 2^{\sigma k}. 
  \end{cases}
\end{equation}
We remark here that since $\bgamma \notin \mathbb{Q}^2$, we have that
$\lim_{k \to \infty} B_k = \infty$,
and thus in particular, $B_k \geq 2$ for $k \geq K_0 = K_0(\bgamma)$. Since
we may assume that $\lim_{q \to \infty} \psi(q) = 0$, we may assume
without loss of generality that $\psi(q) < 2^{-K_0 +1}$ for all
$q \in \mathbb{N}$, which makes $b_k$ and consequently
\eqref{shift_reduced_sets} well-defined.

\subsection{Reduction to a key lemma}\label{sec:reduction-to-key-lemma}

In this subsection, we prove Theorem~\ref{thm:actualresult} assuming
Lemma~\ref{lem:logtolerance_} given below. \th~\ref{lem:logtolerance_}
is the key to establishing quasi-independence on average, and its
proof is the focus of the subsequent sections.

Define
\begin{align}
  E_q'= E_q'(\psi,\bgamma) &= \set*{\balpha \in [0,1]^2 : \abs{q\balpha - \bu - \bgamma} < \psi(q) \textrm{ for some } \bu\in S_q} \label{def_Eq'}\\
                   &= \bigcup_{\bu\in S_q} B\parens*{\frac{\bu + \bgamma}{q}, \frac{\psi(q)}{q}} \cap [0,1]^2,
\end{align}
where $S_q$ is as in~\eqref{shift_reduced_sets}.

\begin{lemma}\th\label{lem:logtolerance_}
  Let $\psi: \NN \to [0, 1/2]$. For $R\leq Q$,
  \begin{equation}\label{eq:logtol}
    \sum_{q\in (Q, 2Q]} \sum_{r\in (R, 2R]}\lambda(E_q'\cap E_r') \\
    \ll \sum_{q\in (Q, 2Q]} \lambda(E_q') \sum_{r\in (R, 2R]}\lambda(E_r') + \frac{\log^+ (Q/R)}{(Q/R)} \parens*{\sum_{q\in (Q, 2Q]}\lambda(E_q') +  \sum_{r\in (R, 2R]}\lambda(E_r')}
  \end{equation}
  where $\log^+$ denotes $\max\set{1, \log_2}$.
\end{lemma}

We will make use of the following recent result of Beresnevich, Velani
and the second-named author~\cite[Theorem\,5]{BHV}:
\begin{proposition}\th\label{full_meas}
  Let $\mu$ be a doubling Borel regular probability measure on a
  metric space $X$. Let $\{A_i\}_{i\in\N}$ be a sequence of
  $\mu$-measurable subsets of $X$.  Suppose that
\begin{equation}\label{eqn01}
\sum_{i=1}^\infty \mu(A_i)=\infty
\end{equation}
and that there exists a constant $C>0$ such that
\begin{equation}\label{eqn02}
\sum_{s,t=1}^Q  \mu(A_s\cap A_t)\le C\left(\sum_{s=1}^Q  \mu(A_s)\right)^2\quad\text{for infinitely many $Q\in\N$\,.}
\end{equation}
 In addition,
suppose that for any $\delta>0$ and any closed ball $B$ centred at $\supp\mu$ there exists $i_0 =i_0 (\delta, B) $ such that for all $i  \geq i_0$,
\begin{equation}\label{vb89}
\mu\left(B\cap A_i\right)\le (1+\delta)\mu\left(B\right)\mu(A_i)\,.
\end{equation}
Then $\mu(\limsup_{q \to \infty}A_q)=1$.
\end{proposition}

We prove two facts about the shift-reduced sets, $E_q'$. The first is
that their measures are comparable to the measures of $E_q$. The
second is that the sets are equidistributed in $[0,1]^2$.

\begin{lemma}\th\label{lem:equid}
  Let $\psi: \mathbb{N} \to [0,1/2]$. Then we have the following:

\begin{itemize}
    \item[(i)] For all $q \geq 1$, $\lambda(E_q')\gg \lambda(E_q)$.
    \item[(ii)] For all balls $B \subseteq [0,1)^2$, we have 
    \begin{equation}
      \label{eq:forBHV}
      \lim_{\substack{q \to \infty\\q \in \supp(\psi)}} \frac{\lambda(B \cap E_q')}{\lambda(B)\lambda(E_q')} = 1.
    \end{equation}
\end{itemize}
\end{lemma}

\begin{proof}
  For (i), note that by $\psi \leq \frac{1}{2}$, we have
  $\lambda(E_q') = 4(\psi(q)/q)^2\#S_q$.  Defining
\[\pi_q := \begin{cases} \prod\limits_{{p \mid q,\;p \nmid B_k}} p, &\text{ if } B_k \leq 2^{\sigma k},\\
\prod\limits_{{p \mid q,\;p \nmid b_k}} p &\text{ if } B_k > 2^{\sigma k},
\end{cases}\]
    we immediately obtain that (see e.g. \cite[Lemma 1]{Kim_new})
    
    \begin{equation}\label{eq:meas_eq}\#S_q = q^2\prod_{\substack{p \mid \pi_q}} \left(1 - \frac{1}{p^2}\right) \geq q^2\prod_{p \in \mathbb{P}} \left(1 - \frac{1}{p^2}\right) = q^2\frac{6}{\pi^2}
    \gg 1,
    \end{equation}
    and (i) follows.

    For (ii), we can argue similarly to \cite[Proposition 1]{BHV} or
    \cite[Lemma 19]{MR} where equidistribution in the $1$-dimensional
    case was shown, but we include a proof for the two-dimensional
    setup for the sake of completeness. Note that \eqref{eq:forBHV}
    follows immediately once we have shown that for every
    $\by = (y_1,y_2) \in [0,1)^2$, we have
    \[\#\{\bu \in S_q:\bu/q \preceq \by\} \sim \#S_q y_1y_2,\quad q \to \infty,\]
where $(x_1,x_2) \preceq (y_1,y_2)$ is defined as $x_1 \leq y_1$ and $x_2 \leq y_2$.
    Note that
\[\begin{split}
\# \Big\{\bu \in S_q: \bu/q \preceq \by \Big\}
\;&=\; \#\big\{0 \preceq \bu \preceq q\by: \gcd(\underbrace{\bA_k+ B_k\bu}_{\bb},B_kq) = 1\big\}
\\&=  \#\big\{0 \preceq \bb \preceq \bA_k + q B_k\by: \bb \equiv
\bA_k \pmod {B_k},\;\; \gcd(\bb,B_kq) = 1\big\}.
\end{split}
\]
Since $\bA_k$ is coprime to $B_k$, we have that
\[
\left\{\begin{array}{r}
\bb\equiv \bA_k \pmod {B_k}\,, \\[1ex]  \gcd(\bb,qB_k) = 1\,,
\end{array}\right.
\;\iff\;
\left\{\begin{array}{l}
\bb\equiv \bA_k \pmod {B_q}\,,  \\[1ex]  \gcd(\bb,\pi_q) = 1\,.
\end{array}\right.
\]
By a standard inclusion-exclusion argument involving the M\"obius
function $\mu$, we can therefore write
\begin{equation}\label{eq99}
\# \Big\{\bu \in S_q: \bu/q \preceq \by \Big\}= \sum_{d \mid \pi_q}\mu(d)\sum_{\substack{0 \preceq \bb \preceq \by qB_k+ \bA_k\\ d\mid \bb}}  \bone_{[\bb \equiv \bA_k \pmod {B_k}]}\,.
\end{equation}
Since $\gcd (d,B_k) \leq \gcd(\pi_q,B_k) = 1$, we have that
\begin{equation}\label{eq100}
\sum_{\substack{0 \preceq \bb < \by q B_k + \bA_k\\ d\mid \bb}}  \bone_{[\bb \equiv \bA_k \pmod {B_k}]} = \sum_{\substack{0 \preceq \bb < \by q B_k + \bA_k\\ \bb \equiv \hat{\bb} \pmod {d B_k}}}   1,
\end{equation}
where $\hat{\bb} = (\hat{b}_1,\hat{b_2})$ is the unique solution $\pmod{B_kd}$ to the system 
$$
\bb \equiv 0 \pmod d,\qquad \bb \equiv \bA_k \pmod{B_k}\,,
$$
which exists by the Chinese Remainder Theorem.  Thus we get
\begin{equation}\label{eq101}\begin{split}
\sum_{\substack{0 \preceq \bb \preceq \by q B_k + \bA_k\\ \bb \equiv \hat{\bb} \pmod {d B_k}}}   1 
= \sum_{\substack{0 \leq b_1 \leq y_1 q B_k + A_{k,1}\\ b_1 \equiv \hat{b_1} \pmod {d B_k}}}
 \sum_{\substack{0 \leq b_2 \leq y_2 q B_k + A_{k,2}\\ b_2 \equiv \hat{b_2} \pmod {d B_k}}}
1 &= \left(
\frac{y_1 q B_k}{B_kd} + O(1)\right)\left(
\frac{y_2 q B_k}{B_kd} + O(1)\right)
\\& = \frac{y_1y_2 q^2}{d^2} + O\left(\frac{q}{d}\right).
\end{split}
\end{equation}
Thus by \eqref{eq99}, we get
\[\begin{split}\#\Big\{\bu \in S_q: \bu/q \preceq \by\Big\} 
&= y_1y_2 q^2 \sum_{d \mid \pi_k}\frac{\mu(d)}{d^2} + O\left(q \sum_{d \mid \pi_k}\frac{\mu(d)}{d}\right)
\\&= y_1y_2 \prod_{p \mid \pi_k}\left(1 - \frac{1}{p^2}\right) + O(q^{3/2})
\\&= y_1y_2\#S_q + O(q^{3/2}),
\end{split}
\]
using \eqref{eq:meas_eq} and a crude divisor-bound.
\end{proof}

\begin{proof}[Proof of Theorem \ref{thm:actualresult}, assuming Lemma \ref{lem:logtolerance_}]
   By Proposition \ref{prop:wlog}, we may assume ~(\ref{divergent}),~(\ref{psi_to_0}),
  and~(\ref{irrational}). We aim to apply Proposition~\ref{full_meas} to prove ~\th~\ref{thm:actualresult}; it remains to verify that for $A_i = E_i'$, we have \eqref{eqn01}, \eqref{eqn02} and \eqref{vb89}. Note that \eqref{eqn01} follows from~(\ref{divergent}) and Lemma \ref{lem:equid}(i), and
  \eqref{vb89} follows from Lemma \ref{lem:equid}(ii), thus we are left to prove \eqref{eqn02}.
 Taking $Q=2^U$ for some $U\in \NN$,
  \begin{align}
    \sum_{q, r\leq 2^U} \lambda(E_q'\cap E_r')
    &\ll \sum_{q\leq 2^U} \sum_{r\leq q}\lambda(E_q'\cap E_r') \nonumber \\
    &\ll \sum_{k=0}^{U-1}\sum_{q\in (2^k, 2^{k+1}]} \sum_{\ell = 0}^k \sum_{r\in (2^\ell,2^{\ell+1}]}\lambda(E_q'\cap E_r')\nonumber \\
    &= \sum_{k=0}^{U-1}\sum_{\ell = 0}^k \underbrace{\sum_{q\in (2^k, 2^{k+1}]} \sum_{r\in (2^\ell,2^{\ell+1}]}\lambda(E_q'\cap E_r')}.\label{eq:4}
  \end{align}
 Using Lemma \ref{lem:logtolerance_}, we get

\begin{align*}
  \sum_{q, r\leq 2^U}
  & \lambda(E_q'\cap E_r') \\
  &\ll \sum_{k=0}^{U-1}\sum_{\ell = 0}^k \brackets*{\sum_{q\in (2^k, 2^{k+1}]} \lambda(E_q') \sum_{r\in (2^\ell,2^{\ell+1}]}\lambda(E_r') + \frac{\log^+ 2^{k-\ell}}{2^{k-\ell}} \parens*{\sum_{q\in (2^k,2^{k+1}]}\lambda(E_q')  + \sum_{r\in (2^\ell,2^{\ell+1}]}\lambda(E_r')}} \\
  &\ll \sum_{q\leq 2^U}\sum_{r \leq 2^U}\lambda(E_q')\lambda(E_r') \\
  &\quad + \sum_{k=0}^{U-1}\sum_{q\in (2^k,2^{k+1}]}\lambda(E_q') \underbrace{\sum_{\ell = 0}^k \frac{\log^+ 2^{k-\ell}}{2^{k-\ell}}}_{\ll 1} + \sum_{\ell=0}^{U-1}\sum_{r\in (2^\ell,2^{\ell+1}]}\lambda(E_r') \underbrace{\sum_{k = \ell}^{U-1} \frac{\log^+ 2^{k-\ell}}{2^{k-\ell}}}_{\ll 1} \\
  &\ll \parens*{\sum_{q\leq 2^U}\lambda(E_q')}^2 + \sum_{q\leq 2^U}\lambda(E_q') + \sum_{r\leq 2^U}\lambda(E_r') 
  \ll \parens*{\sum_{q\leq 2^U}\lambda(E_q')}^2,
\end{align*}
with the last inequality following from~(\ref{divergent}). Since $U\in\NN$ is arbitrary, this
establishes~(\ref{eqn02}) as needed.
\end{proof}

The rest of this article is devoted to proving Lemma \ref{lem:logtolerance_}.

\section{Overlap estimates and Bohr set bounds}

This section contains the precise estimates that are needed in the
proof of~\th~\ref{lem:logtolerance_}. In the overlap estimates of
\th~\ref{lem:overlaps} below, one can see reflected the case
distinctions that led to the various refinements of the model problem
in Section~\ref{sec:heur-model-probl}. In the Bohr set bounds
of~\th~\ref{adhocbohrsetbounds} below, we carry out general
two-dimensional versions of the computations that led to the solution
of the model problem.

\subsection{Overlap estimates for the shift-reduced sets}

\begin{lemma}[Overlap estimates]\th\label{lem:overlaps}
  Let $\psi:\NN\to [0, 1/2]$ and $\bgamma\in\RR^2$. For $r, q\in \NN$, put
  \begin{gather*}
   \Delta := \max\set*{\frac{2\psi(q)}{q}, \frac{2\psi(r)}{r}}, \quad
 \delta := \min\set*{\frac{2\psi(q)}{q}, \frac{2\psi(r)}{r}},\\
    D:= \max\set*{\frac{2\psi(q)}{q}, \frac{2\psi(r)}{r}}\frac{qr}{\gcd(q,r)}, \\
    q'= \frac{q}{\gcd(q,r)}, \qquad r'= \frac{r}{\gcd(q,r)}, \qquad h = q'-r',
  \end{gather*}
  Suppose $q\in (Q, 2Q]$, $r\in (R, 2R]$, $r< q$,
  $2^{-k} < \psi(q) \leq 2^{-k+1}$ and
  $2^{-\ell}< \psi(r) \leq 2^{-\ell+1}$,
  and set
  \begin{equation*}
    \hatD:=\hatD(h) = \max\set*{D, h\abs*{\bgamma - \frac{\bA_\bullet}{B_\bullet}}}, \qquad\textrm{where}\qquad    \bullet =
    \begin{cases}
      k &\textrm{if } \frac{\psi(q)}{q} > \frac{\psi(r)}{r}, \\
      \ell &\textrm{if } \frac{\psi(q)}{q} \leq \frac{\psi(r)}{r}.
    \end{cases}
  \end{equation*}
  The following hold:
  \begin{itemize}
  \item For all $r,q$ as above,
      \begin{equation*}
        \lambda(E_q'\cap E_r') \ll \delta^2\gcd(q,r)^2(D^2+1)
        \ll \lambda(E_q')\lambda(E_r') + \frac{1}{D^2}\lambda(E_q')\lambda(E_r').
      \end{equation*}
 
  \item If $r\nmid q$ and
    $B_\bullet \leq \max\set{2^\bullet D,
      2^{\sigma\bullet}}$, then
      \begin{equation*}
    \lambda(E_q'\cap E_r') \ll
      \lambda(E_q')\lambda(E_r') + 
      \frac{1}{D^2}\lambda(E_q')\lambda(E_r')\bone_{\brackets*{\norm*{h\bgamma}<D}}\parens*{\bone_{\brackets*{2\widehat{D} \geq \frac{1}{B_\bullet}}} \bone_{\brackets*{B_\bullet \leq 2^{\sigma\bullet}}} + \bone_{\brackets*{2^{\sigma\bullet} < B_\bullet}}}.
    \end{equation*}
    
  \item If $r\nmid q$ and $B_\bullet > \max\set{2^\bullet D, 2^{\sigma\bullet}}$,
    then
  \begin{equation*}
    \lambda(E_q'\cap E_r') \ll
    \lambda(E_q')\lambda(E_r') 
    + 
      \frac{1}{D^2}\lambda(E_q')\lambda(E_r')\bone_{\brackets*{\norm*{h\bgamma}<D}}\parens*{\bone_{\brackets*{2D \geq \frac{1}{b_\bullet}}}\bone_{\brackets*{h\norm{b_\bullet\bgamma}\leq 1/2}} + \bone_{\brackets*{h\norm{b_\bullet\bgamma}>1/2}}}.
    \end{equation*}
  \end{itemize}
\end{lemma}

\begin{proof}
  Trivially, we have
  \begin{align*}
    \lambda(E_q'\cap E_r')
    &\leq \delta^2 \#\set*{(\bu,\bv)\in S_q\times S_r : \abs*{\frac{\bu + \bgamma}{q} - \frac{\bv + \bgamma}{r}} < \Delta} \\
    &= \delta^2 \#\set*{(\bu,\bv)\in S_q\times S_r : \abs{q'(\bv + \bgamma)- r'(\bu + \bgamma)} < D(q,r)}.
  \end{align*}

  Each value taken by
  \begin{equation*}
    \bc(\bu, \bv):=q'(\bv + \bgamma) -r'(\bu + \bgamma) 
  \end{equation*}
  is achieved by $\gcd(q,r)^2$ pairs
  $(\bu', \bv')\in \ZZ_q^2\times\ZZ_r^2$. Consequently, 
  \begin{align*}
    \lambda(E_q'\cap E_r')
    &\leq \delta^2 \#\set*{(\bu,\bv)\in S_q\times S_r : \abs{\bc(\bu,\bv)} < D} \\
    &\leq \delta^2\gcd(q,r)^2\#\underbrace{\set*{\bc(\bu,\bv) : (\bu,\bv)\in S_q\times S_r \textrm{ and } \abs{\bc(\bu,\bv)} < D}}_{C(q,r)}.
  \end{align*}
  Since 
  $C(q,r) \subset h\bgamma + \ZZ^2$, we have
  \begin{equation*}
    \#C(q,r) \ll (D^2 + 1)\bone_{\brackets*{C(q,r)\neq \emptyset}},
  \end{equation*}
  leading to
\begin{align*}
 \lambda(E_q'\cap E_r') &\ll \delta^2\gcd(q,r)^2(D^2+1)\bone_{\brackets*{C(q,r)\neq \emptyset}}\\
                        &\ll \lambda(E_q)\lambda(E_r) + \frac{1}{D^2}\lambda(E_q)\lambda(E_r)\bone_{\brackets*{C(q,r)\neq \emptyset}}\bone_{\brackets*{D <1}} \\
                        &\ll \lambda(E_q')\lambda(E_r') + \frac{1}{D^2}\lambda(E_q')\lambda(E_r')\bone_{\brackets*{C(q,r)\neq \emptyset}}\bone_{\brackets*{D <1}},
\end{align*}
by~\th\ref{lem:equid}(i). In the remainder of the proof, we will
assume that $r \nmid q$, and will identify necessary conditions for
$C(q,r)\neq \emptyset$ when $D<1$.
  
\subsubsection*{Case 1:
  $B_\bullet \leq \max\set{2^\bullet D, 2^{\sigma\bullet}}$}

Supposing $\abs{\bc(\bu,\bv)}<D$, the triangle inequality gives
\begin{align}
    \abs*{r'\parens*{\bu + \frac{\bA_\bullet}{B_\bullet}} - q'\parens*{\bv + \frac{\bA_\bullet}{B_\bullet}}}
    &< D + \abs*{r'\parens*{\bgamma-\frac{\bA_\bullet}{B_\bullet}} - q'\parens*{\bgamma - \frac{\bA_\bullet}{B_\bullet}}}\label{eq:lhs} \\
    &\leq 2\hatD. \nonumber
\end{align}
If $B_\bullet \leq 2^{\sigma\bullet}$, the coprimality condition
from $(\bu,\bv)\in S_q\times S_r$ (recall \eqref{shift_reduced_sets}), and the fact that $r\nmid q$ show
that the left-hand side of~\eqref{eq:lhs} is bounded below by
$\frac{1}{B_\bullet}$, so we arrive at
\begin{equation*}
  \bone_{\brackets*{C(q,r)\neq\emptyset}}\bone_{\brackets*{D <1}}\bone_{\brackets*{B_\bullet\leq 2^{\sigma\bullet}}}
  \leq \bone_{\brackets*{2\hatD \geq \frac{1}{B_\bullet}}}\bone_{\brackets*{\norm*{h\bgamma} < D}}.
\end{equation*}
If, on the other hand,
$2^{\sigma\bullet} < B_\bullet \leq 2^\bullet D$, then there
is no coprimality condition enforcing an upper bound on $1/B_\bullet$,
so
\begin{equation*}
  \bone_{\brackets*{C(q,r)\neq\emptyset}}\bone_{\brackets*{D <1}}\bone_{\brackets*{2^{\sigma\bullet} < B_\bullet \leq 2^{\bullet} D}}
  \leq \bone_{\brackets*{\norm*{h\bgamma} < D}}\bone_{\brackets*{2^{\sigma\bullet} < B_\bullet}}
\end{equation*}
in this case.

\subsubsection*{Case 2: $B_\bullet > \max\set{2^{\bullet} D ,2^{\sigma\bullet}}$}

For $D < 1$ the condition $\abs{\bc(\bu,\bv)}<D$ implies
\begin{equation}\label{eq:Dtau}
  \abs*{q'(\bv + \bgamma) - r'(\bu + \bgamma)} < D,
\end{equation}
and thus in particular
\begin{equation}\label{eq:Dtaunorm}
  \norm*{h\bgamma} < D.
\end{equation}
On the other hand, the triangle inequality and shift-reducedness with respect to $b_{\bullet}$ gives
\begin{align*}
  \abs*{q'\parens*{\bv + \bgamma} - r'\parens*{\bu + \bgamma}}
  &\geq \abs*{\frac{1}{b_\bullet} - \abs*{(q'-r')\parens*{\bgamma-\frac{\ba_\bullet}{b_\bullet}}}}\\
  &= \frac{1}{b_\bullet} \abs*{1 - h \norm*{b_\bullet\bgamma}}.
\end{align*}
Consequently, if $h\norm{b_\bullet\bgamma} \leq \frac{1}{2}$,
then~\eqref{eq:Dtau} requires
\begin{equation*}
  \frac{1}{b_\bullet} < 2D.
\end{equation*}
On the other hand, if $h \norm{b_\bullet\bgamma} > \frac{1}{2}$,
then there is no such requirement. We have
\begin{equation}
  \bone_{\brackets*{C(q,r)\neq \emptyset}}\bone_{\brackets*{D <1}}
  \leq \bone_{\brackets*{\norm*{h\bgamma}<D}}\parens*{\bone_{\brackets*{2D \geq \frac{1}{b_\bullet}}}\bone_{\brackets*{h\norm{b_\bullet\bgamma}<1/2}} + \bone_{\brackets*{h\norm{b_\bullet\bgamma}\geq 1/2}}}
\end{equation}
in this case.
\end{proof}

\subsection{Bohr set bounds}\label{sec:bohr-set-bounds}

The main result of this section is~\th~\ref{adhocbohrsetbounds}. It
contains the Bohr set bound exactly as it is applied in the proof
of~\th~\ref{lem:logtolerance_}. \th~\ref{bohrsetbound,lem:oncearound}
are simpler computations that are used in various parts of the proof
of~\th~\ref{adhocbohrsetbounds}.

\begin{lemma}\th\label{bohrsetbound} Let $(\ba,b)\in\ZZ^2\times\NN$
  with $\gcd(\ba, b)=1$, and $\bbeta = (\beta_1,\beta_2)\in\RR^2$. Then for all $0 < \varepsilon < 1$, 
  \begin{equation*}
    \sum_{0\leq h \leq b-1}\bone_{\brackets*{\norm*{h\parens*{\frac{\ba}{b}} +
          \bbeta}<\eps}} \ll \eps b + 1,
  \end{equation*}
  and for all $N\geq 1$,
  \begin{equation*}
  \sum_{1\leq h \leq N}\bone_{\brackets*{\norm*{h\parens*{\frac{\ba}{b}}}<\eps}} \ll (\eps
  b + 1)\frac{N}{b}\bone_{\brackets*{N\geq b}} + \min\set{N,\eps
    b}\bone_{\brackets*{N<b}},
\end{equation*}
where the implicit constants are absolute.
\end{lemma}

\begin{proof}
  Let $\ba=(a_1,a_2)$, $d=\gcd(a_1,b)$, $a_1'=a_1/d$, and $b' =
  b/d$. Then
  \begin{align}
    \sum_{0\leq h \leq
        b-1}\bone_{\brackets*{\norm*{h\parens*{\frac{\ba}{b}} +
            \bbeta}<\eps}} 
    &\leq \sum_{h=0}^{b'-1}\bone_{\brackets*{\norm*{h\parens*{\frac{a_1}{b}} + \beta_1}<\eps}}\cdot\sum_{k= 0}^{d-1}\bone_{\brackets*{\norm*{(h+kb')\parens*{\frac{a_2}{b}} + \beta_2}<\eps}}
      \nonumber\\
    &\leq (2\eps b' + 1)(2\eps d+1) \nonumber \\
    &\leq 8\eps b + 1. \label{eq:fullorbit}
  \end{align}
  This proves the first assertion.

 In particular, when $\bbeta=0$, we have shown that
  $\#S\cap (-\eps, \eps)^2 \leq 8\eps b + 1$, where
  \begin{equation*}
    S = \set*{h\parens*{\frac{\ba}{b}}\mod 1 : h = 0, 1, \dots, b-1}\subset\TT^2.
  \end{equation*}
  Now, if $N<b$, then $\set{h(\ba/b)\pmod 1: h=1, \dots, N}$ is a
  subset of $S$ omitting the point $\bzero$ (which always lies in
  $S \cap (-\varepsilon,\varepsilon)$), and we obtain the bound
  \begin{equation*} \sum_{1\leq h \leq
N}\bone_{\brackets*{\norm*{h\parens*{\frac{\ba}{b}}}<\eps}} \leq
\min\set{N, 8 \eps b}\qquad (N < b)
\end{equation*} from~(\ref{eq:fullorbit}). (The bound $\leq N$ is
trivial in this case.) On the other hand, if $N\geq b$, then
$\set{h(\ba/b)\pmod 1: h=1, \dots, N}$ runs through all residue
classes at most $N/b + 1 \leq 2N/b$ times, and~(\ref{eq:fullorbit})
leads to the bound
\begin{equation*}
  \sum_{1\leq h \leq N}\bone_{\brackets*{\norm*{h\parens*{\frac{\ba}{b}}}<\eps}}
  \leq(8\eps b + 1) \frac{2N}{b} \qquad (N \geq b).
\end{equation*}
The second assertion of the lemma is a combination of these two
statements.
\end{proof}

\begin{lemma}\th\label{lem:oncearound}
  For all $\balpha = (\alpha_1, \alpha_2)\in\RR^2\setminus\QQ^2$, $(\beta_1, \beta_2) \in\RR^2$, and
  $\eps>0$ we have
  \begin{equation*}
    \sum_{1\leq m <1/\norm{\balpha}}\ind{\norm{m\balpha + \bbeta} < \eps} \leq \frac{2\eps}{\norm{\balpha}} + 1. 
  \end{equation*}
\end{lemma}

\begin{proof}
 Without loss of generality, we may
  assume that $\norm{\balpha}=\norm{\alpha_1}$. Since for all $m$ we
  have $\norm{m\balpha + \bbeta}\geq \norm{m\alpha_1+\beta}$, we obtain
  \begin{align*}
    \sum_{1\leq m <1/\norm{\balpha}}\ind{\norm{m\balpha + \bbeta} < \eps}
    &\leq \sum_{1\leq m <1/\norm{\alpha_1}}\ind{\norm{m\alpha_1 + \beta_1} < \eps} \\
    &= \#\parens*{\underbrace{\set*{(m\alpha_1 + \beta_1) \mod 1 : 1 \leq m < \frac{1}{\norm{\alpha_1}}}}_S \cap \brackets*{(-\eps,\eps)\mod 1}}.
      \intertext{The points in $S$ have a minimum separation of $\norm{\alpha_1}$ on
      the circle, therefore we continue with}
    &\leq \frac{2\eps}{\norm{\alpha_1}} + 1 = \frac{2\eps}{\norm{\balpha}} + 1,
  \end{align*}
  as required.
\end{proof}

\begin{lemma}\th\label{adhocbohrsetbounds}
  Under the hypotheses and notation of~\th~\ref{lem:overlaps}, let
    \begin{multline}
    \label{eq:bone_C}
    \bone_C = \ind{B_\bullet \leq \max\set{2^\bullet D, 2^{\sigma\bullet}}}\parens*{\bone_{\brackets*{2\widehat{D} \geq \frac{1}{B_\bullet}}} \bone_{\brackets*{B_\bullet \leq 2^{\sigma\bullet}}} + \bone_{\brackets*{2^{\sigma\bullet} < B_\bullet}}} \\
    +\ind{B_\bullet > \max\set{2^\bullet D, 2^{\sigma\bullet}}}\parens*{\bone_{\brackets*{2D \geq \frac{1}{b_\bullet}}}\bone_{\brackets*{h\norm{b_\bullet\bgamma}\leq 1/2}} + \bone_{\brackets*{h\norm{b_\bullet\bgamma}>1/2}}}
  \end{multline}
  and
  \begin{equation}\label{eq:Ndefinition}
    N \asymp
    \begin{cases}
      2^\bullet D &\textrm{if } \bullet = \ell, \\
      \frac{Q}{R} 2^\bullet D &\textrm{if } \bullet = k.
    \end{cases}
  \end{equation}
  Then there exists $\omega >0$ such that
  \begin{equation*}
    \frac{1}{N}\sum_{1\leq h\leq N}\ind{\norm{h\bgamma} < D}\bone_C \ll D^\omega + \ind{D^{\omega}\gg \frac{R}{Q}}\ind{\bullet = k}   + \ind{D\gg 1}\ind{\bullet = \ell}
  \end{equation*}
  holds for all $D < 1$.
\end{lemma}

\begin{remark*}
  The proof reveals that one can take $\omega = 1/6$.  We also point
  out that in the above estimate, the $\ll$ symbols in the indicators
  only depend on the implied constants arising in the definition of
  $N$ in \eqref{eq:Ndefinition}, that is, if
  $\frac{1}{c}2^{\bullet}D \leq N \leq c2^{\bullet}D$ for some
  $c > 0$, then
  \begin{equation*}
    \frac{1}{N}\sum_{1\leq h\leq N}\ind{\norm{h\bgamma} < D}\bone_C
    \ll D^\omega + \ind{D^{\omega}\geq
      C_1\parens*{\frac{R}{Q}}}\ind{\bullet = k} + \ind{D\geq
      C_2}\ind{\bullet = \ell}
  \end{equation*}
  where $C_1$ and $C_2$ are positive constants depending only on $c$.
\end{remark*}

\begin{proof}
  We proceed in cases.
  
  \subsection*{Case 1:
    $B_\bullet \leq \max\set{2^{\bullet}D ,2^{\sigma\bullet}}$}

  Fix a parameter $\tau \in (\sigma, 1)$ and split the sum:
  \begin{align}
    \frac{1}{N}\sum_{1 \leq h\leq N}\bone_{\brackets*{\norm*{h \bgamma}<D}}\bone_{C}
    &= \frac{1}{N}\parens*{\sum_{\substack{1 \leq h \leq N \\ \hatD(h)\leq D^\tau}}\bone_{\brackets*{\norm*{h \bgamma}<D}} + \sum_{\substack{1< h\leq N \\ \hatD(h)>D^\tau}}\bone_{\brackets*{\norm*{h \bgamma}<D}}}\bone_C \nonumber \\
    &\leq \frac{1}{N}\parens*{\sum_{\substack{1 \leq h \leq N \\ \hatD(h)\leq D^\tau}}\bone_{\brackets*{\norm*{h \bgamma}<D^\tau}} + \sum_{\substack{1< h\leq N \\ \hatD(h)>D^\tau}}\bone_{\brackets*{\norm*{h \bgamma}<D}}}\bone_C. \label{eq:jan2nd}
  \end{align}

  \subsubsection*{Case 1a: First term of~(\ref{eq:jan2nd})}
  By the triangle inequality,
  \begin{equation}\label{eq:triangle}
   \norm*{h\parens*{\frac{\bA_\bullet}{B_\bullet}}}  \leq \norm{h\bgamma}  + \abs*{h(\bgamma - (\bA_\bullet/B_\bullet))} \leq \norm{h\bgamma} + \hatD(h),
  \end{equation}
  so the first sum is bounded by
  \begin{align}
    \frac{1}{N}\sum_{\substack{1 \leq h \leq N \\ \hatD(h)\leq D^\tau}}\bone_{\brackets*{\norm*{h \bgamma}<D^\tau}}\bone_C
    &\leq \frac{1}{N}\sum_{1 \leq h \leq N}\bone_{\brackets*{\norm*{h \parens*{\frac{\bA_\bullet}{B_\bullet}}}<2D^\tau}}\bone_C \nonumber\\
    &\ll \frac{1}{N} \parens*{(D^\tau B_\bullet + 1) \frac{N}{B_\bullet} + B_\bullet D^\tau\bone_{\brackets*{N < B_\bullet}}}\bone_C \tag{\th~\ref{bohrsetbound}}\nonumber\\
    &\ll \frac{1}{N} \parens*{(D^\tau B_\bullet + 1) \frac{N}{B_\bullet} + B_\bullet D^\tau\bone_{\brackets*{2^\bullet D \ll B_\bullet}}}\bone_C,\nonumber
      \intertext{recalling that $N \asymp 2^\bullet D$ or  $N \asymp \frac{Q}{R}2^\bullet D$. Recalling $\hatD(h)\leq D^\tau$ and bounding $\bone_C$ accordingly, we continue with}
    &\ll  \parens*{D^\tau + \frac{1}{B_\bullet} + \frac{B_\bullet D^\tau}{N}\ind{2^\bullet D \ll B_\bullet}}\parens*{\bone_{\brackets*{2D^\tau \geq \frac{1}{B_\bullet}}} \bone_{\brackets*{B_\bullet \leq 2^{\sigma\bullet}}} + \bone_{\brackets*{2^{\sigma\bullet} < B_\bullet}}}\nonumber\\
    &\ll \parens*{D^\tau + \frac{1}{B_\bullet} }\parens*{\bone_{\brackets*{2D^\tau \geq \frac{1}{B_\bullet}}} \bone_{\brackets*{B_\bullet \leq 2^{\sigma\bullet}}} + \bone_{\brackets*{2^{\sigma\bullet} < B_\bullet}}} + \frac{2^{\sigma\bullet}D^\tau}{2^\bullet D}\nonumber\\
    &\ll D^\tau + 2^{-\sigma\bullet}+ 2^{\bullet(\sigma-1)}D^{\tau-1} \nonumber\\
    &\ll D^{\sigma}+ D^{\tau-\sigma},\label{eq:COMB1}
  \end{align}
  having noted that the definition of $D$ implies
  \begin{equation}\label{eq:kgeqj}
    2^{-\bullet} \leq D.
  \end{equation}

\subsubsection*{Case 1b: Second term in~(\ref{eq:jan2nd})}

The second sum in~\eqref{eq:jan2nd} is nonempty only if
\begin{equation*}\label{eq:8jan2}
  D^\tau < \hatD(h) = h \abs*{\bgamma - \frac{\bA_\bullet}{B_\bullet}} \leq \frac{N}{B_\bullet}\norm{B_\bullet\bgamma}.
\end{equation*}
Combining this with~(\ref{eq:Bkbound}) and~(\ref{eq:Ndefinition}) gives
\begin{equation}\label{eq:case1b}
  D \gg
  \begin{cases}
    1 &\textrm{if } \bullet= \ell, \\
    \parens*{\frac{R}{Q}}^{\frac{1}{1-\tau}} &\textrm{if } \bullet= k,
  \end{cases}
\end{equation}
where the implicit constant depends only on the implicit constant
appearing in~(\ref{eq:Ndefinition}).

  \subsection*{Case 2: $B_\bullet>\max\set{2^{\bullet} D,2^{\sigma\bullet}}$}

  Fix a parameter $0 < \rho < \sigma/2$.
  
  \subsubsection*{Case 2a: $N\geq b_\bullet$ and $2b_\bullet >D^{-\rho}$}

  By writing
  $h=mb_\bullet + s, 1 \leq s \leq b$, we have
  $h\bgamma = mb_\bullet\bgamma + s(\ba_\bullet/b_\bullet) + s(\bgamma
  - \ba_\bullet/b_\bullet)$, and
  \begin{align*}
    s\abs*{\bgamma - \frac{\ba_\bullet}{b_\bullet}}
    \leq \frac{1}{B_\bullet^{1/2}},
  \end{align*}
  and the sum is bounded by
  \begin{align}
    \frac{1}{N}\sum_{1\leq h \leq N}\bone_{\brackets*{\norm*{h\bgamma} < D}}
    &\leq \frac{1}{N}\sum_{0\leq m \leq \frac{N}{b_\bullet}}\sum_{s=1}^{b_\bullet}\bone_{\brackets*{\norm*{mb_\bullet\bgamma + s\parens*{\frac{\ba_\bullet}{b_\bullet}}} < D + B_\bullet^{-1/2}}} \nonumber\\
    &\ll \frac{1}{N}\parens*{\frac{N}{b_\bullet} +1}\parens*{b_\bullet D + b_\bullet B_\bullet^{-1/2}+1} \tag{\th~\ref{bohrsetbound}} \\
    &\ll D + \frac{1}{B_\bullet^{1/2}}+\frac{1}{b_\bullet} \nonumber\\
    &\ll 2^{-\sigma\bullet/2} +D^\rho \leq D^\rho \label{eq:COMB6}.
  \end{align}

  \subsubsection*{Case 2b: $N\geq b_\bullet$ and $2b_\bullet \leq D^{-\rho}$}
  Assume now that $N\geq b_\bullet$ and $2b_\bullet \leq D^{-\rho}$.
  Split the sum
  \begin{equation}\label{eq:jan2netflix}
    \frac{1}{N}\sum_{1\leq h \leq N}\bone_{\brackets*{\norm*{h\bgamma} < D}}\bone_C
    \ll\frac{1}{N}\parens*{\sum_{\substack{1\leq h \leq N \\ h\norm{b_\bullet \bgamma} \leq 1/2}}\bone_{\brackets*{\norm*{h\bgamma} < D}}
    + \sum_{\substack{1\leq h \leq N \\ h\norm{b_\bullet \bgamma} > 1/2}}\bone_{\brackets*{\norm*{h\bgamma} < D}}}\bone_C.
  \end{equation}
  By~\th~\ref{lem:overlaps}, the condition
  $h\norm{b_\bullet\bgamma}\leq 1/2$ applied in $\bone_C$ implies
  $\frac{1}{b_\bullet} \leq 2 D$, so the first sum is empty. The
  second sum in~\eqref{eq:jan2netflix} is nonempty only if
  \begin{equation}\label{big1/2}
    \frac{1}{2} < h\norm{b_\bullet\bgamma} \leq N \norm{b_\bullet\bgamma}.
  \end{equation}
    \th~\ref{lem:oncearound} implies that among any
    $\ceil{1/\norm{b_{\bullet}\bgamma}}$ consecutive values of $m$
    there are at most
  \begin{equation*}
    \frac{2 (D+ B_\bullet^{-1/2})}{\norm{b_{\bullet}\bgamma}}+2 \overset{(\ref{eq:bk})}{\ll} \frac{D+ B_\bullet^{-1/2}}{\norm{b_{\bullet}\bgamma}}
  \end{equation*}
  for which 
  \begin{equation*}
    \ind{\norm*{m b_\bullet \bgamma + s\parens*{\frac{\ba_\bullet}{b_\bullet}}} < D+ B_\bullet^{-1/2}}=1,
  \end{equation*}
  regardless of $s$. With these observations and arguing as in the previous case, the second sum
  in~(\ref{eq:jan2netflix}) is bounded by
  \begin{align}
    \frac{1}{N}\sum_{\substack{1\leq h \leq N \\ h\norm{b_\bullet \bgamma} > 1/2}}\bone_{\brackets*{\norm*{h\bgamma} < D}}\bone_C
    &\leq \frac{1}{N}\sum_{s=1}^{b_\bullet}\sum_{0\leq m \leq \frac{N}{b_\bullet}}\bone_{\brackets*{\norm*{m(b_{\bullet}\bgamma) + s\parens*{\frac{\ba_\bullet}{b_\bullet}}} <  D+ B_\bullet^{-1/2}}} \nonumber\\
    &\ll \frac{b_\bullet}{N}\parens*{\frac{N}{b_\bullet}\norm{b_\bullet\bgamma} + 1}\frac{ D+ B_\bullet^{-1/2}}{\norm{b_\bullet\bgamma}} \nonumber \\
    &\ll \parens*{1 + \frac{b_\bullet}{N\norm{b_\bullet \bgamma}}}( D+ B_\bullet^{-1/2}) \nonumber\\
    &\overset{\eqref{big1/2}}{\ll} b_\bullet\parens*{D+ B_\bullet^{-1/2}} \nonumber\\
    &\ll D^{1-\rho} + D^{-\rho}2^{-\sigma\bullet/2} \leq D^{1-\rho} + D^{\frac{\sigma}{2}-\rho}.\label{eq:COMB7}
  \end{align}

  \subsubsection*{Case 2c: $N< b_\bullet$}
   In this case, for all $1 \leq h \leq N$ we have
  \begin{align*}
    \norm*{h\parens*{\frac{\ba_\bullet}{b_\bullet}}}
    &=\norm*{h\bgamma - h\parens*{\bgamma - \frac{\ba_\bullet}{b_\bullet}}} \\
    &\leq\norm*{h\bgamma} + \abs*{h\parens*{\bgamma - \frac{\ba_\bullet}{b_\bullet}}} \\
    &\leq \norm*{h\bgamma} + \frac{N}{b_\bullet B_\bullet^{1/2}},
  \end{align*}
  using \eqref{eq:bk}. Thus we have
  \begin{align}
    \frac{1}{N}\sum_{1\leq h \leq N} \ind{\norm*{h\bgamma}<D}
    &\leq \frac{1}{N}\sum_{1\leq h \leq N} \ind{\norm*{h\parens*{\frac{\ba_\bullet}{b_\bullet}}}<D + \frac{N}{b_\bullet B_\bullet^{1/2}}}\nonumber  \\
    &= \frac{1}{N}\parens*{\min\set*{N, Db_\bullet + \frac{N}{B_\bullet^{1/2}}}} \tag{\th~\ref{bohrsetbound}} \\
    &= \min\set*{1, \frac{Db_\bullet}{N} + \frac{1}{B_\bullet^{1/2}}} \nonumber \\
    &\ll \frac{1}{B_\bullet^{1/2}} \leq 2^{-\sigma\bullet/2} \overset{(\ref{eq:kgeqj})}{\leq} D^{\sigma/2}\label{eq:COMB5},
  \end{align}
  by $N\gg 2^\bullet D$ and~(\ref{eq:bkbound}), as well as the case
  assumption that $B_\bullet > 2^{\sigma\bullet}$.

  The lemma follows upon
  combining~(\ref{eq:COMB1}),~(\ref{eq:case1b}),~(\ref{eq:COMB5}),~(\ref{eq:COMB6}),
  and~(\ref{eq:COMB7}), and taking $\omega>0$ to be the smallest power
  of $D$ that appears, or $1-\tau$, whichever is smaller.
\end{proof}

\begin{remark*}
  For concreteness, we may choose $\sigma = 2/3$, $\tau = 5/6$, and
  $\rho = 1/6$, in which case we obtain $\omega = 1/6$.
\end{remark*}

\section{Proof of
  Lemma~\ref{lem:logtolerance_}}\label{sec:proof-lemma-refl}

Let $R\leq Q$. For $j \geq 1$, let
\begin{equation*}
  D_j = \set*{(q,r)\in \parens*{(Q,2Q]\times (R, 2R]}\setminus F : D(q,r)\in (2^{-j}, 2^{-j+1}], r< q},
\end{equation*}
put
\begin{equation*}
  D_0 = \set*{(q,r)\in (Q,2Q]\times (R, 2R] : D(q,r)>1, r< q},
\end{equation*}
and
\begin{equation*}
  F = \set*{(q,r)\in (Q, 2Q]\times (R, 2R] : r < q, r\mid q,\, D(q,r)\leq 1}.
\end{equation*}
We have
\begin{gather*}
  \lambda(E_q'\cap E_r') \ll \lambda(E_q')\lambda(E_r') \qquad ((q,r)\in D_0) \\
  \lambda(E_q'\cap E_r') \ll \parens*{\frac{r}{q}}^2\lambda(E_q') \qquad ((q,r)\in F)
\end{gather*}
by~\th~\ref{lem:overlaps}. Then
\begin{multline}
  \sum_{q\in (Q, 2Q]} \sum_{\substack{r\in (R, 2R] \\ r < q}}\lambda(E_q'\cap E_r')
  = \sum_{j\geq 0}\sum_{(q,r)\in D_j} \lambda(E_q'\cap E_r')  + \sum_{(q,r)\in F}\lambda(E_q'\cap E_r') \\
  \ll \sum_{q\in (Q, 2Q]}\lambda(E_q')\sum_{r\in (R, 2R]}\lambda(E_r') + \sum_{j\geq 1}\sum_{(q,r)\in D_j} \lambda(E_q'\cap E_r') + \sum_{(q,r)\in F}\lambda(E_q'\cap E_r'), \label{eq:prestar}
\end{multline}
and
\begin{align}
  \sum_{(q,r)\in F}\lambda(E_q'\cap E_r')
  &\ll \sum_{q\in (Q, 2Q]} \sum_{\substack{r\in (R, 2R] \\ r < q, r\mid q}}\parens*{\frac{r}{q}}^2\lambda(E_q') \nonumber \\
  &\leq \sum_{q\in (Q, 2Q]} \parens*{\sum_{\substack{d\in [q/(2R), q/R) \\ d\mid q}}\frac{1}{d^2}}\lambda(E_q') \nonumber \\
  &\leq \sum_{q\in (Q, 2Q]} \parens*{\sum_{n\geq \frac{Q}{2R}}\frac{1}{n^2}}\lambda(E_q') \nonumber \\
  &\ll \frac{R}{Q}\sum_{q\in (Q, 2Q]}\lambda(E_q'). \label{eq:F}
\end{align}
Using~\th~\ref{lem:overlaps} again, we have
\begin{align}
  \sum_{j\geq 1}\sum_{(q,r)\in D_j} \lambda(E_q'\cap E_r')
  &\ll \sum_{j\geq 1}\sum_{(q,r)\in D_j} \lambda(E_q')\lambda(E_r') \nonumber \\
  &\qquad + \sum_{j\geq 1}\sum_{(q,r)\in D_j} \frac{1}{D^2}\lambda(E_q')\lambda(E_r')\bone_{\brackets*{\norm*{\frac{q-r}{\gcd(q,r)}\bgamma}<D}}\bone_C\label{eq:star}
\end{align}
where $\bone_C$ is as in~(\ref{eq:bone_C}). The first sum on the
right-hand side is an acceptable contribution to~(\ref{eq:prestar}),
and it only remains to bound the expression~(\ref{eq:star}) in the
second sum. We will split it over the sets
\begin{align*}
  D_j^1 &= \set*{(q,r)\in D_j : \frac{\psi(r)}{r} < \frac{\psi(q)}{q}, \psi(r) \leq \psi(q)}, \\
  D_j^2 &= \set*{(q,r)\in D_j : \frac{\psi(r)}{r} \geq \frac{\psi(q)}{q}, \psi(r) \geq \psi(q)}, \textrm{ and} \\
  D_j^3 &= \set*{(q,r)\in D_j : \frac{\psi(r)}{r} \geq \frac{\psi(q)}{q}, \psi(r) < \psi(q)}.
\end{align*}
Denote those collections $\calD_i = \set{D_j^i}_{j\geq 1}\, (i=1,2,3)$.

\subsection{Case 1: (\ref{eq:star}) over $\calD_1$}

For $(q,r)\in D_j^1$ we have
\begin{equation*}
  \psi(r) \leq \frac{r}{q}\psi(q) \leq \frac{2R}{Q}\psi(q)
\end{equation*}
so we may bound
\begin{multline*}
  \sum_{j\geq 1}\sum_{(q,r)\in D_j^1} \frac{1}{D^2}\lambda(E_q')\lambda(E_r')\bone_{\brackets*{\norm*{\frac{q-r}{\gcd(q,r)}\bgamma}<D}}\bone_C \\
  \ll \sum_{j\geq 1}\sum_{k\geq 1} \sum_{\ell \geq k + L-2}\sum_{\substack{(q,r)\in D_j^1 \\ \psi(q) \in \left(2^{-k}, 2^{-k+1}\right] \\ \psi(r)\in \left(2^{-\ell}, 2^{-\ell+1}\right]}}\frac{1}{D^2}\lambda(E_q')\lambda(E_r')\bone_{\brackets*{\norm*{\frac{q-r}{\gcd(q,r)}\bgamma}<D}}\bone_C \\
  \ll \sum_{j\geq 1}\sum_{k\geq 1} \sum_{\ell \geq k+L-2} 2^{2j-2k-2\ell} \sum_{\substack{(q,r)\in D_j^1 \\ \psi(q) \in \left(2^{-k}, 2^{-k+1}\right] \\ \psi(r)\in \left(2^{-\ell}, 2^{-\ell+1}\right]}}\bone_{\brackets*{\norm*{\frac{q-r}{\gcd(q,r)}\bgamma}<D}}\bone_C,
\end{multline*}
where $L = \log^+ (Q/R)$ and $\bone_C$ is as
  in~(\ref{eq:bone_C}).  Further, since
\begin{equation*}
  D = \frac{\psi(q)r}{\gcd(q,r)}\in (2^{-j}, 2^{-j+1}],
\end{equation*}
we have
\begin{equation*}
  \frac{r}{\gcd(q,r)} \in \left(2^{k-j-1}, 2^{k-j+1}\right]
\end{equation*}
whenever $\psi(q) \in (2^{-k}, 2^{-k+1}]$, which implies that we may
restrict to $k\geq j-1$ and bound the expression above by

\begin{multline*}
  \ll \sum_{j\geq 1}\sum_{k\geq j-1} \sum_{\substack{q\in (Q,2Q] \\ \psi(q)\in \left(2^{-k}, 2^{-k+1}\right]}}\sum_{\ell \geq k+L-2} 2^{2j-2k-2\ell} \sum_{\substack{r \in (R, 2R] \\ \frac{r}{\gcd(q,r)} \in \left(2^{k-j-1}, 2^{k-j+1}\right] \\ \psi(r)\in (2^{-\ell}, 2^{-\ell+1}]}}\bone_{\brackets*{\norm*{\frac{q-r}{\gcd(q,r)}\bgamma}<2^{-j+1}}}\bone_C \\
  \ll \sum_{j\geq 1}\sum_{k\geq j- 1} \sum_{\substack{q\in (Q,2Q] \\ \psi(q)\in \left(2^{-k}, 2^{-k+1}\right]}}\sum_{\ell \geq k + L - 2} 2^{2j-2k-2\ell}\sum_{e \in \left(2^{k-j-1}, 2^{k-j+1}\right]}\sum_{\substack{r \in (R, 2R] \\ \frac{r}{\gcd(q,r)} =e}}\bone_{\brackets*{\norm*{\frac{q-r}{\gcd(q,r)}\bgamma}<2^{-j+1}}} \bone_C.
\end{multline*}
For $q\in (Q, 2Q], r\in (R,2R]$ with
\begin{equation*}
  \frac{r}{\gcd(q,r)} \in \left(2^{k-j-1}, 2^{k-j+1}\right]
\end{equation*}
and $r<q$, we have
\begin{equation*}
  1\leq \abs*{\frac{q-r}{\gcd(q,r)}} \leq \frac{Q}{R}2^{k-j+2}.
\end{equation*}
Furthermore, for each fixed $q\in (Q, 2Q]$ and
$e\in (2^{k-j-1}, 2^{k-j+1}]$, the values of
\begin{equation*}
  h = \frac{q-r}{\gcd(q,r)}
\end{equation*}
are distinct as $r$ ranges through $(R, 2R]$. Bounding the $e$-sum
trivially by $\ll 2^{k-j}$ and extending the $h$-sum over its maximal
range, we have
\begin{align*}
  &\ll \sum_{j\geq 1}\sum_{k\geq j- 1} \sum_{\substack{q\in (Q,2Q] \\ \psi(q)\in \left(2^{-k}, 2^{-k+1}\right]}}\sum_{\ell \geq k + L - 2} 2^{j-k-2\ell}\sum_{1 \leq h \leq (Q/R)2^{k-j+2}}\bone_{\brackets*{\norm*{h\bgamma}<2^{-j+1}}} \bone_C \nonumber \\
  &\ll \parens*{\frac{Q}{R}} \sum_{j\geq 1}\sum_{k\geq j - 1} \sum_{\substack{q\in (Q,2Q] \\ \psi(q)\in \left(2^{-k}, 2^{-k+1}\right]}}\sum_{\ell \geq k + L - 2} 2^{-2\ell}  \underbrace{\parens*{\frac{1}{N}\sum_{1\leq h\leq N}\bone_{\brackets*{\norm*{h\bgamma}<2^{-j+1}}} \bone_C}}_{N:=(Q/R)2^{k-j+2}} \\
    &\ll \parens*{\frac{Q}{R}} \sum_{j\geq 1}\sum_{k\geq j - 1} \sum_{\substack{q\in (Q,2Q] \\ \psi(q)\in \left(2^{-k}, 2^{-k+1}\right]}}\sum_{\ell \geq k + L - 2} 2^{-2\ell} \parens*{2^{-\omega j} + \ind{2^{-j+1}\gg\parens*{\frac{R}{Q}}^{1/\omega}}}, 
\end{align*}
by~\th~\ref{adhocbohrsetbounds}. Summing over $\ell$ gives
\begin{align}
  \ll &\parens*{\frac{R}{Q}}\sum_{j\geq 1} \sum_{k\geq j-1}  \parens*{2^{-\omega j} + \ind{2^{-j+1}\gg\parens*{\frac{R}{Q}}^{1/\omega}}}\sum_{\substack{q\in (Q,2Q] \\ \psi(q)\in \left(2^{-k}, 2^{-k+1}\right]}}\lambda(E_q') \nonumber \\
  &\ll \parens*{\frac{R}{Q}}\sum_{k\geq 1}\parens*{\sum_{j\geq 1} 2^{-\omega j}}\sum_{\substack{q\in (Q,2Q] \\ \psi(q)\in \left(2^{-k}, 2^{-k+1}\right]}}\lambda(E_q') \nonumber\\
 &\qquad + \parens*{\frac{R}{Q}}\sum_{k\geq 1}\parens*{\sum_{j\geq 1}  \ind{2^{-j+1}\gg\parens*{\frac{R}{Q}}^{1/\omega}}}\sum_{\substack{q\in (Q,2Q] \\ \psi(q)\in \left(2^{-k}, 2^{-k+1}\right]}}\lambda(E_q') \nonumber \\
  &\ll \parens*{\frac{RL}{Q}} \sum_{q\in (Q,2Q]}\lambda(E_q'), \label{eq:D1punchline}
\end{align}
recalling that the convention $\log^+ = \max\set{1,\log_2}$ implies that
$L\geq 1$.

\subsection{Case 2: (\ref{eq:star}) over $\calD_2$}

For $(q,r)\in D_j^2$ we have
\begin{equation*}
  \psi(r) \geq  \psi(q) \quad\textrm{and}\quad \bullet=\ell, 
\end{equation*}
so we may bound
\begin{multline*}
  \sum_{j\geq 1}\sum_{(q,r)\in D_j^2} \frac{1}{D^2}\lambda(E_q')\lambda(E_r')\bone_{\brackets*{\norm*{\frac{q-r}{\gcd(q,r)}\bgamma}<D}}\bone_C \\
  \ll \sum_{j\geq 1}\sum_{\ell \geq 1}\sum_{k\geq \ell}\sum_{\substack{(q,r)\in D_j^2 \\ \psi(r)\in (2^{-\ell}, 2^{-\ell + 1}] \\ \psi(q) \in (2^{-k}, 2^{-k + 1}]}} \frac{1}{D^2}\lambda(E_q')\lambda(E_r')\bone_{\brackets*{\norm*{\frac{q-r}{\gcd(q,r)}\bgamma}<D}}\bone_C \\
                  \ll \sum_{j\geq 1}\sum_{\ell \geq 1}\sum_{k\geq \ell} 2^{2j-2k-2\ell} \sum_{\substack{(q,r)\in D_j^2 \\ \psi(q) \in (2^{-k}, 2^{-k+1}] \\ \psi(r) \in (2^{-\ell}, 2^{-\ell+1}]}}\bone_{\brackets*{\norm*{\frac{q-r}{\gcd(q,r)}\bgamma}<D}}\bone_C.
\end{multline*}
Now we have
\begin{equation*}
  D = \frac{\psi(r)q}{\gcd(q,r)} \in (2^{-j}, 2^{-j+1}],
\end{equation*}
and for $\psi(r)\in (2^{-\ell}, 2^{-\ell+1}]$,
\begin{equation*}
  \frac{q}{\gcd(q,r)} \in \left(2^{\ell-j-2}, 2^{\ell-j+2}\right]\qquad\textrm{and}\quad\frac{r}{\gcd(q,r)} \in \left(\frac{R}{Q}2^{\ell-j-2}, \frac{R}{Q}2^{\ell-j+2}\right]. 
\end{equation*}
We continue with
\begin{equation*}
  \ll \sum_{j\geq 1}\sum_{\ell \geq 1}\sum_{k \geq\ell}\sum_{\substack{r\in (R, 2R] \\ \psi(r) \in (2^{-\ell}, 2^{-\ell+1}]}} 2^{2j-2k-2\ell} \sum_{e \leq \frac{R}{Q}2^{\ell-j+2}}\sum_{\substack{q\in (Q, 2Q] \\ \frac{r}{\gcd(q,r)}= e}}\bone_{\brackets*{\norm*{\frac{q-r}{\gcd(q,r)}\bgamma}<2^{-j+1}}}\bone_C.
\end{equation*}
For $q\in (Q, 2Q], r\in (R,2R]$ with
\begin{equation*}
  \frac{r}{\gcd(q,r)} \leq  \frac{R}{Q}2^{\ell-j+2}
\end{equation*}
we have
\begin{equation*}
  1 \leq \abs*{\frac{q-r}{\gcd(q,r)}} \leq 2^{\ell-j+3}.
\end{equation*}
Furthermore, for each fixed $r\in (R, 2R]$ and $e\leq (R/Q)2^{\ell-j+2}$, the
values of
\begin{equation*}
  h = \frac{q-r}{\gcd(q,r)}
\end{equation*}
are distinct as $q$ ranges through $(Q, 2Q]$. Bounding the $e$-sum
trivially by $\ll (R/Q)2^{\ell-j}$ and extending the $h$-sum over its maximal
range, we have
\begin{align*}
  &\ll \parens*{\frac{R}{Q}}\sum_{j\geq 1}\sum_{\ell\geq 1}\sum_{k\geq \ell}\sum_{\substack{r\in (R, 2R] \\ \psi(r) \in (2^{-\ell}, 2^{-\ell+1}]}}  2^{j-\ell-2k}\sum_{1 \leq h \leq 2^{\ell-j+3}}\bone_{\brackets*{\norm*{h\bgamma}<2^{-j+1}}}\bone_C \\
  &\ll \parens*{\frac{R}{Q}}\sum_{j\geq 1}\sum_{\ell\geq 1}\sum_{k\geq \ell}\sum_{\substack{r\in (R, 2R] \\ \psi(r) \in (2^{-\ell}, 2^{-\ell+1}]}}  2^{-2k}\underbrace{\parens*{\frac{1}{N}\sum_{1 \leq h \leq N}\bone_{\brackets*{\norm*{h\bgamma}<2^{-j+1}}}\bone_C}}_{N=2^{\ell-j+3}} \\
  &\ll \parens*{\frac{R}{Q}}\sum_{j\geq 1}\sum_{\ell\geq 1}\sum_{k\geq \ell}\sum_{\substack{r\in (R, 2R] \\ \psi(r) \in (2^{-\ell}, 2^{-\ell+1}]}}  2^{-2k}\parens*{2^{-\omega j} + \ind{2^{-j+1} \gg 1}}
\end{align*}
by~\th~\ref{adhocbohrsetbounds}. Summing over $k$ and rearranging
gives
\begin{equation}\label{eq:D2punchline}
  \ll \parens*{\frac{R}{Q}}\parens*{\sum_{j\geq 1}\parens*{2^{-\omega j} + \ind{2^{-j+1} \gg 1}}}\sum_{\ell\geq 1} \sum_{\substack{r\in (R, 2R] \\ \psi(r) \in (2^{-\ell}, 2^{-\ell+1}]}} \psi(r)^2
  \ll \parens*{\frac{R}{Q}}\sum_{r\in (R, 2R]} \lambda(E_r').
\end{equation}

\subsection{Case 3: (\ref{eq:star}) over $\calD_3$}

For $(q,r)\in D_j^3$ we have
\begin{equation*}
  \frac{R}{2Q}\psi(q) \leq \psi(r) \leq \psi(q), \quad\textrm{and}\quad \bullet = \ell,
\end{equation*}
so we may bound
\begin{multline*}
  \sum_{j\geq 1}\sum_{(q,r)\in D_j^3} \frac{1}{D^2}\lambda(E_q')\lambda(E_r')\bone_{\brackets*{\norm*{\frac{q-r}{\gcd(q,r)}\bgamma}<D}}\bone_C \\
  \ll \sum_{j\geq 1}\sum_{k \geq 1}\sum_{\ell = k}^{k + L+2} \sum_{\substack{(q,r)\in D_j^3 \\ \psi(q) \in (2^{-k}, 2^{-k+1}] \\ \psi(r) \in (2^{-\ell}, 2^{-\ell+1}]}} \frac{1}{D^2}\lambda(E_q')\lambda(E_r') \bone_{\brackets*{\norm*{\frac{q-r}{\gcd(q,r)}\bgamma}<D}}\bone_C \\
  \ll \sum_{j\geq 1}\sum_{k \geq 1}\sum_{\ell = k}^{k + 2 +  L} 2^{2j-2k-2\ell}\sum_{\substack{(q,r)\in D_j^3 \\ \psi(q) \in (2^{-k}, 2^{-k+1}] \\ \psi(r) \in (2^{-\ell}, 2^{-\ell+1}]}}\bone_{\brackets*{\norm*{\frac{q-r}{\gcd(q,r)}\bgamma}<D}}\bone_C.
\end{multline*}
Now we have
\begin{equation*}
  D = \frac{\psi(r)q}{\gcd(q,r)} \in (2^{-j}, 2^{-j+1}],
\end{equation*}
and for $\psi(r)\in (2^{-\ell}, 2^{-\ell+1}]$,
\begin{equation*}
  \frac{q}{\gcd(q,r)} \in \left(2^{\ell-j-2}, 2^{\ell-j+2}\right]\qquad\textrm{and}\quad\frac{r}{\gcd(q,r)} \in \left(\frac{R}{Q}2^{\ell-j-2}, \frac{R}{Q}2^{\ell-j+2}\right].
\end{equation*}
We continue with 
\begin{equation*}
  \ll \sum_{j\geq 1}\sum_{k \geq 1}\sum_{\ell = k}^{k + 2 +  L}\sum_{\substack{q\in (Q, 2Q] \\ \psi(q) \in (2^{-k}, 2^{-k+1}]}} 2^{2j-2k-2\ell} \sum_{e \leq \frac{R}{Q}2^{\ell-j+2}}\sum_{\substack{r\in (R, 2R] \\ \frac{r}{\gcd(q,r)}= e}}\bone_{\brackets*{\norm*{\frac{q-r}{\gcd(q,r)}\bgamma}<2^{-j+1}}}\bone_C.
\end{equation*}
For each fixed $q\in (Q,2Q]$ and $e \leq (R/Q)2^{\ell-j+2}$, the values of
\begin{equation*}
  h = \frac{q-r}{\gcd(q,r)}
\end{equation*}
are distinct as $r$ ranges through $(R, 2R]$. Bounding the $e$-sum
trivially by $\ll (R/Q)2^{\ell-j}$ and extending the $h$-sum over its maximal
range, we have
\begin{align*}
  &\ll  \parens*{\frac{R}{Q}} \sum_{j\geq 1}\sum_{k \geq 1}\sum_{\ell = k}^{k + 2 +  L}\sum_{\substack{q\in (Q, 2Q] \\ \psi(q) \in (2^{-k}, 2^{-k+1}]}} 2^{j-2k-\ell} \sum_{1 \leq h \leq 2^{\ell-j+3}}\bone_{\brackets*{\norm*{h\bgamma}<2^{-j+1}}}\bone_C \\
  &\ll \parens*{\frac{R}{Q}}\sum_{j\geq 1}\sum_{k\geq 1}\sum_{\ell = k}^{k+2+ L}\sum_{\substack{q\in (Q, 2Q] \\ \psi(q) \in (2^{-k}, 2^{-k+1}]}}  2^{-2k}\underbrace{\parens*{\frac{1}{N}\sum_{1\leq h \leq N}\bone_{\brackets*{\norm*{h\bgamma}<2^{-j+1}}}\bone_C}}_{N=2^{\ell-j+3}} \\
  &\ll \parens*{\frac{R}{Q}}\sum_{j\geq 1}\sum_{k\geq 1}\sum_{\ell = k}^{k+2+ L}\sum_{\substack{q\in (Q, 2Q] \\ \psi(q) \in (2^{-k}, 2^{-k+1}]}}  2^{-2k}\parens*{2^{-\omega j} + \ind{2^{-j+1} \gg 1}}
\end{align*}
by~\th~\ref{adhocbohrsetbounds}. Rearranging gives
\begin{multline}\label{eq:D3punchline}
  \ll \parens*{\frac{R}{Q}}\sum_{k\geq 1} \parens*{\sum_{j\geq 1}\parens*{2^{-\omega j} + \ind{2^{-j+1} \gg 1}}}\parens*{\sum_{\ell=k}^{k+2+ L} 1} \sum_{\substack{q\in (Q, 2Q] \\ \psi(q) \in (2^{-k}, 2^{-k+1}]}}\psi(q)^2 \\
  \ll \parens*{\frac{RL}{Q}}\sum_{q\in (Q, 2Q]}\lambda(E_q').
\end{multline}
By combining~(\ref{eq:D1punchline}),~(\ref{eq:D2punchline}),
and~(\ref{eq:D3punchline}), we find that
\begin{equation*}
  \sum_{j\geq 1}\sum_{(q,r)\in D_j} \lambda(E_q'\cap E_r') \ll \frac{\log^+(Q/R)}{Q/R}\parens*{\sum_{q\in (Q, 2Q]}\lambda(E_q') +  \sum_{r\in (R, 2R]}\lambda(E_r')}.
\end{equation*}
Putting this and~(\ref{eq:F}) into~(\ref{eq:prestar})
proves~\th~\ref{lem:logtolerance_}.

\subsection*{Acknowledgments}

MHT was funded by the Austrian Science Fund, FWF project 10.55776/ESP5134624. This research was also supported by funding from DA's London Mathematical Society Emmy Noether Fellowship, grant number EN-2425-05.
We thank Christoph Aistleitner and Sam Chow for comments on an earlier version of this manuscript.

\bibliographystyle{plain}

\bibliography{main.bib}

\end{document}